\documentclass[11pt,reqno]{amsart}

\usepackage{amsmath}
\usepackage{amsfonts}
\usepackage{amssymb}
\usepackage{graphics}
\usepackage{graphicx}
\usepackage{color}
\usepackage{hyperref}
\usepackage{cite}
\usepackage{txfonts}
\usepackage{geometry}

\theoremstyle{plain}
\newtheorem{theorem}{Theorem}[section]
\newtheorem*{theorem*}{Theorem}

\newtheorem{lemma}{Lemma}[section]
\newtheorem{corollary}{Corollary}[section]
\theoremstyle{definition}

\theoremstyle{remark}
\newtheorem{remark}{Remark}[section]
\newtheorem{example}{Example}[section]

\numberwithin{equation}{section}

\newcommand{\e}{^\varepsilon}
\newcommand{\eps}{{\varepsilon}}
\newcommand{\ds}{\displaystyle}
\newcommand{\cupl}{\bigcup\limits}
\newcommand{\suml}{\sum\limits}
\newcommand{\intl}{\int\limits}
\newcommand{\liml}{\lim\limits}
\newcommand{\maxl}{\max\limits}
\newcommand{\minl}{\min\limits}
\newcommand{\A}{A}

\renewcommand{\a}{\alpha}
\renewcommand{\b}{\beta}

\begin{document}

\author[Andrii Khrabustovskyi]{}
\title[Opening up and control of spectral gaps of the Laplacian in periodic domains]{} 

\noindent{\LARGE\bf Opening up and control of spectral gaps of the Laplacian in periodic domains}\bigskip

\noindent{\large\bf Andrii Khrabustovskyi}\bigskip

{\small
\noindent {Research Training Group "Analysis, Simulation and Design of Nanotechnological Processes"\\
Department of Mathematics, Karlsruhe Institute of Technology

\noindent{andrii.khrabustovskyi@kit.edu}\bigskip\smallskip

\noindent{\bf Abstract.} The main result of this work is as follows: for arbitrary pairwise
disjoint finite intervals $(\alpha_j,\beta_j)\subset[0,\infty)$,
$j=1,\dots,m$ and for arbitrary $n\geq 2$ we construct a family
of periodic non-compact domains
$\{\Omega^\varepsilon\subset\mathbb{R}^n\}_{\varepsilon>0}$ such
that the spectrum of the Neumann Laplacian in $\Omega^\varepsilon$
has at least $m$ gaps when $\varepsilon$ is small enough, moreover
the first $m$ gaps tend to the intervals $(\alpha_j,\beta_j)$ as
$\varepsilon\to 0$. The constructed domain $\Omega^\varepsilon$ is
obtained by removing from $\mathbb{R}^n$ a system of periodically
distributed "trap-like" surfaces.
\textcolor{black}{The parameter $\varepsilon$ characterizes the period 
of the domain $\Omega\e$, also it is involved in a geometry of the removed 
surfaces.}

\medskip

}\smallskip

\normalsize
\section*{Introduction}

The problem we are going to solve belongs to the spectral theory of periodic 
self-adjoint differential operators. It is known that usually the spectrum of 
such operators is a locally finite union of compact intervals called 
\textit{bands}. In general the bands may overlap. The open interval 
$(\a,\b)\subset\mathbb{R}$ is called a \textit{gap} if it has an empty 
intersection with the spectrum, but its ends belong to it.

In general the presence of gaps is not guaranteed, for example, the spectrum of 
the Laplacian in $L_2(\mathbb{R}^n)$ has no gaps: 
$\sigma(-\Delta_{\mathbb{R}^n})=[0,\infty)$. Therefore one of the central 
questions arising here is whether the gaps really exist in concrete situations. 
This question is motivated by various applications, in particular in connection 
with photonic crystals attracting much attention in recent years. Photonic 
crystals are periodic dielectric media in which electromagnetic waves of certain 
frequencies cannot propagate, which is caused by gaps in the spectrum of the 
Maxwell operator or related scalar operators. We refer to paper 
\cite{Kuchment_PC} concerning mathematical problems arising in this field.

The problem of constructing of periodic operators with spectral gaps attracts a 
lot attention in the last twenty years. Various examples were presented in 
\cite{Khrab2,Figotin1,Figotin2,Figotin3,Frielander,Hempel,Zhikov,Hoang} for 
periodic divergence type elliptic operators in $\mathbb{R}^n$, in 
\cite{HempelHerbst} for periodic Schr\"{o}dinger operators, in 
\cite{Figotin2,Filonov} for Maxwell operators with periodic coefficients in 
$\mathbb{R}^n$, in \cite{DavHar,Post,Green,Exner,Khrab1} for Laplace-Beltrami 
operators on periodic Riemannian manifolds,  in \cite{Nazarov2} for  Laplacians 
in periodic domains in $\mathbb{R}^n$. We refer to overview \cite{HempelPost} 
where these and other related
questions are discussed in more details. Also we mention papers 
\textcolor{black}{\cite{BorPan1,BorPan2,Pank,Nazarov1,Cardone1,Cardone2,Yoshi,
Bakharev,BBC,Cardone3,FriSol,Nazarov3,Nazarov4,Nazarov5}} devoted to the same 
problem for the operators posed in unbounded domains with a waveguide geometry 
(quantum waveguides). 

The present paper is devoted to spectral analysis of the Neumann Laplacians in 
periodic domains. We denote by $\mathcal{H}_n$ the set of all domains
$\Omega\subset\mathbb{R}^n$ satisfying the property
$$\exists d=d(\Omega)>0:\ \Omega=\Omega+d k,\ \forall k\in \mathbb{Z}^n$$
(i.e. $\Omega$ is periodic and the cube $(-d/2,d/2)^n$ is a period cell). Let 
$\Omega\in \mathcal{H}_n$ and $\mathcal{A}$ be the Neumann
Laplacian in $\Omega$. Operators of this type occur in various areas of physics. 
For example in the case $n=2$ the operator $\mathcal{A}$ governs the propagation 
of $H$-polarized electro-magnetic waves in a periodic dielectric medium with a 
perfectly conducting boundary. Below (see Remark \ref{remappl}) we discuss an 
application of our results to the theory of $2D$ photonic crystals.

The example of periodic domain
with gaps in the spectrum of the Neumann Laplacian was presented
in \cite{Nazarov2}. Here the authors considered the Neumann
Laplacian  in  $\mathbb{R}^2$ perforated by
$\mathbb{Z}^2$-periodic family of circular holes and proved that
the gaps in its spectrum open up when the diameter of holes is
close enough to the distance between their centers (the last one
is fixed). 

In the present work we want not only to construct a new type of
periodic domains with gaps in the spectrum of the
Neumann Laplacian but also be able to control the
edges of these gaps making them close (in some natural sense) to predefined 
intervals. Let us formulate our main result.

\begin{theorem}[Main Theorem]\label{th0}
Let $L>0$ be an arbitrarily large number and let $(\a_j,\b_j)$
($j={1,\dots,m},\ m\in \mathbb{N} $) be arbitrary intervals
satisfying
\begin{gather}\label{intervals}
0<\a_1,\quad \a_j<\b_{j}< \a_{j+1},\ j=\overline{1,m-1},\quad
\a_m<\b_{m}<L.
\end{gather}
Let $n\in\mathbb{N}\setminus\{1\}$.

Then one can construct the family of domains
$\left\{\Omega\e\in \mathcal{H}_n\right\}_{\eps>0}$ \textcolor{black}{with 
$d(\Omega\e)=\eps$} such
that the spectrum of the Neumann Laplacian in $\Omega\e$
(we denote it $\mathcal{A}\e$) has the following structure in the interval 
$[0,L]$ when $\eps$ is
small enough:
\begin{gather}\label{spec1}
\sigma(\mathcal{A}\e)\cap[0,L]=[0,L]\setminus
\left(\cupl_{j=1}^{m}(\a_j^\eps,\b_j^\eps)\right),
\end{gather}
where the intervals $(\a_j\e,\b_j\e)$ satisfy
\begin{gather}\label{spec2}
\forall j=1,\dots,m:\quad \liml_{\eps\to 0}\a_j^\eps=\a_j,\
\liml_{\eps\to 0}\b_j^\eps=\b_j.
\end{gather}
\end{theorem}

\begin{remark} In work \cite{CDV} Y. Colin de Verdi\`{e}re proved (among other
results) the following statement: for arbitrary numbers
$0=\lambda_1<\lambda_2<\dots <\lambda_m<\infty$ ($m\in\mathbb{N}$)
there exists a bounded domain $\Omega\subset \mathbb{R}^n$ ($n\geq
2$) such that the first $m$ eigenvalues of the Neumann Laplacian
in $\Omega$ are exactly $\lambda_1,\dots,\lambda_m$.  Our theorem can be 
regarded as an analogue
of this result for the Neumann Laplacians in non-compact periodic domains. 
\end{remark}

Some preliminary results towards the proof of Theorem \ref{th0}
were obtained by the author and E. Khruslov in 
\cite{Khrab3} where the case $m=1$ was considered. However the
general case $m\geq 2$ is much more complicated. Similar results for the 
Laplace-Beltrami operators on periodic
Riemannian manifolds without a boundary and for elliptic operators
in the entire space $\mathbb{R}^n$ were obtained by the author in
\cite{Khrab1} and \cite{Khrab2} correspondingly.

\begin{figure}[h]
\begin{picture}(380,170)      
      
\scalebox{0.9}{\includegraphics{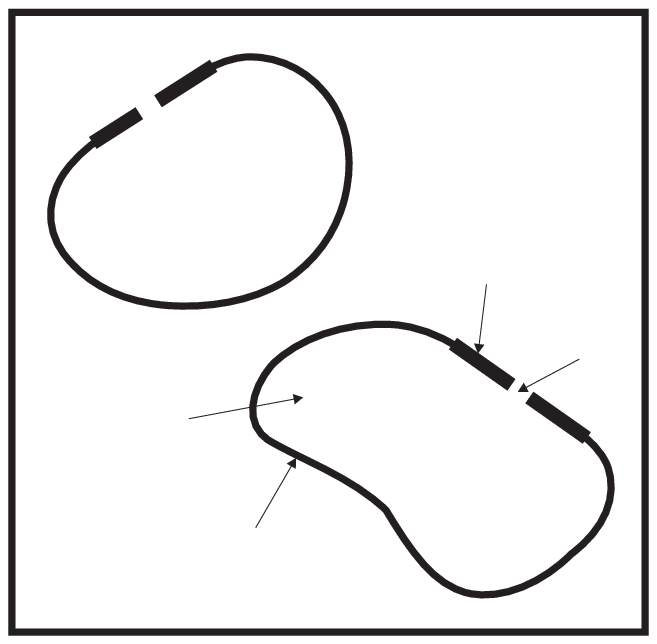}}\hspace{20mm}\scalebox{0.9}{
\includegraphics{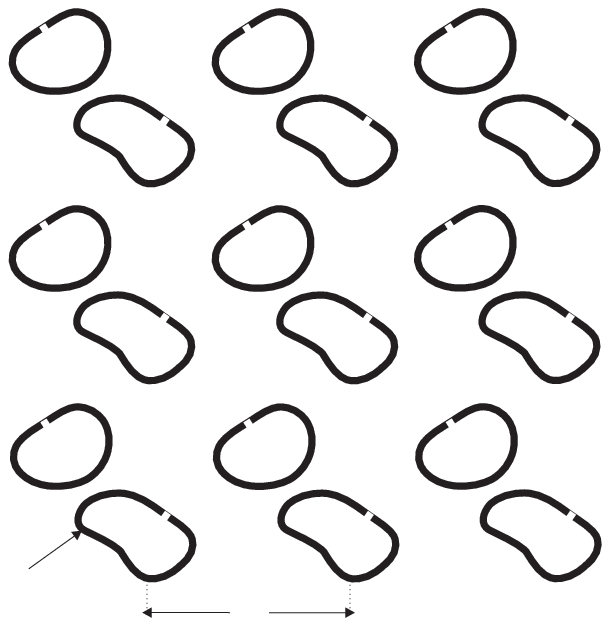}}
      \put(-229, 75){$D_{j}\e$}
      \put(-342, 50){$B_{j}$}
      \put(-327, 18){$S_{j}\e$}
      \put(-161, 10){$S_{ij}\e$}
      \put(-95, 3){$\eps$}
      \put(-285, 95){${\text{flat part of }\partial B_j}$}
    \end{picture}
  \caption{The system of screens $S_{ij}\e$. Here $m=2$.}\label{fig1}
\end{figure}

We now briefly explain how to construct the family
$\left\{\Omega\e\right\}_{\eps>0}$. Let $B_j$, $j=1,\dots,m$ be
pairwise disjoint open domains belonging to the unit cube
$(-1/2,1/2)^n$ in $\mathbb{R}^n$. We suppose that for any
$j=1,\dots,m$ $\partial B_j$ contains a flat part. Within this
flat part we make a small circular hole $D_j\e$, the obtained set
we denote by $S_j\e$ (see the left picture on Fig. \ref{fig1}):
$$S_j\e=\partial B_j\setminus D_j\e,\ j=1,\dots m.$$
Here $\eps>0$ is a parameter characterizing the size of the hole,
namely we suppose that the radius of $D_j\e$ is equal to
$d_j\eps^{2/n-2}$ if $n>2$ or $\exp(-1/d_j\eps^2)$ if $n=2$. Here
$d_j$, $j=1,\dots,m$ are positive constants. Finally we set
\begin{gather*}
\Omega\e=\mathbb{R}^n\setminus\left(\cupl_{i\in
\mathbb{Z}^n}\cupl_{j=1}^m S_{ij}\e\right),\text{ where
}S_{ij}\e=\eps(S_j\e+i),
\end{gather*}
i.e. $\Omega\e$ is obtained by removing from $\mathbb{R}^n$ $m$
families of periodically distributed "trap-like" surfaces (see
Figure \ref{fig1}, right picture). Obviously,
$\Omega\e\in\mathcal{H}_n$, the cube $(-\eps/2,\eps/2)$ is the
period cell. We denote by $\mathcal{A}\e$ the Neumann Laplacian
in $\Omega\e$ (the precise definition will be given in the next
section).

We will prove (see Theorem \ref{th1}) that for an arbitrarily
large interval $[0,L]$ the spectrum of the operator
$\mathcal{A}\e$ has exactly $m$ gaps in $[0,L]$ when $\eps$ is
small enough. Moreover when $\eps\to 0$ these gaps converge to
some intervals $(\sigma_j,\mu_j)$ ($j=1,\dots m$) depending in a
special way on the domains $B_j$ and the numbers $d_j$ and
satisfying
\begin{gather}\label{intervals_sigma}
0<\sigma_1,\quad \sigma_j<\mu_{j}< \sigma_{j+1},\ j=\overline{1,m-1},\quad
\sigma_m<\mu_{m}.
\end{gather}
Finally we will prove (see Lemma \ref{mainlemma}) that  for an
arbitrary intervals $(\a_j,\b_j)$, $j=1,\dots m$ satisfying
\eqref{intervals} one can choose $B_j$ and $d_j$ in such a way
that the equalities
$$\sigma_j=\a_j,\ \mu_j=\b_j,\ j=1,\dots,m$$
hold. For the volumes of the sets $B_j$ and for the numbers $d_j$
we will present the exact formulae. It is clear that the main
theorem follows directly from Theorem \ref{th1} and Lemma
\ref{mainlemma}.

\begin{remark}
The idea how to construct the domain $\Omega\e$  is close to the
idea which was used in \cite{Khrab2}, where the operator
$-(b\e)^{-1}\mathrm{div}(a^{\eps}\nabla)$ in $\mathbb{R}^n$ was
studied. In this work the role of "traps" is played by the family
of thin spherical shells which are $\eps\mathbb{Z}^n$-periodically
distributed in $\mathbb{R}^n$ and on which $a\e(x)$ becomes small
as $\eps\to 0$. A similar idea was also used in \cite{Khrab1}
where the periodic Laplace-Beltrami operator was studied.

The analysis of the asymptotic behaviour of spectra was carried
out in \cite{Khrab1,Khrab2} using the methods of the
homogenization theory. The idea to use this theory in order to
open up the gaps in the spectrum of periodic differential
operators was firstly proposed in \cite{Zhikov}. Since the proof
in \cite{Khrab1,Khrab2} is rather cumbersome, in the present work
we carry out the analysis using another method (see the next
remark). On the other hand the results of \cite{Khrab1, Khrab2}
helped us to guess the form of the equation \eqref{mu_eq} below
whose roots are the limits of the right ends of the spectral
bands.

Boundary value problems in domains with "traps" were also
considered in \cite{Bourgeat1,Bourgeat2}, where the authors
studied the homogenization of semi-linear parabolic equations and
their attractors. Similar homogenization problems were studied in
\cite{March}.

\end{remark}

\begin{remark}
Let us briefly describe the scheme of the proof of Theorem
\ref{th1}. We enclose the left end ({resp.} the right end) of the
$k$-th band between the $k$-th eigenvalues of the Neumann and
periodic ({resp.} the antiperiodic and Dirichlet) Laplacians posed
on the period cell. We prove that both ends of this enclosure
converge to $\mu_{k-1}$ if $k=2,\dots,m+1$ and to infinity if
$k>m+1$ ({resp.} converge to $\sigma_k$ if $k=1,\dots,m$ and to
infinity if $k>m$) as $\eps\to 0$.

The most difficult part of the proof is the investigation of the
asymptotic behaviour of the eigenvalues of the Neumann Laplacian (see Theorem
\ref{Nth}). To obtain the asymptotics of eigenvalues we will
construct convenient approximations for the corresponding
eigenfunctions. The analysis of the eigenvalues of the Dirichlet Laplacian (see
Theorem \ref{Dth}) is carried out using the same ideas but it is
essentially simpler. The analysis of the eigenvalues of the periodic ({resp.}
antiperiodic) Laplacian repeats  word-by-word the analysis for the eigenvalues 
of the
Neumann ({resp.} Dirichlet) Laplacian.

\end{remark}

\begin{remark}\label{remappl} The obtained results can be applied in the theory 
of
$2D$ photonic crystals. Let us introduce the following sets in $\mathbb{R}^3$:
$$\mathbf{\Omega}\e=\left\{(x_1,x_2,z)\in \mathbb{R}^3:\ x=(x_1,x_2)\in 
\Omega\e,\ z\in\mathbb{R}\right\},\quad \mathbf{S}\e=\mathbb{R}^3\setminus 
\mathbf{\Omega}\e,$$
where  $\Omega\e\subset \mathbb{R}^2$ is defined above. We suppose
that $\mathbf{\Omega}\e$ is occupied by a dielectric medium
whereas the union of  the screens $\mathbf{S}\e$ is occupied by a
perfectly conducting material. It is supposed that the electric
permittivity and the magnetic permeability of the material
occupying $\mathbf{\Omega}\e$ are equal to $1$. 

The propagation of electromagnetic waves in $\mathbf{\Omega}\e$ is governed by 
the Maxwell operator $\mathcal{M}\e$ (below by $E$ and $H$ we denote the 
electric and magnetic fields, $U=(E,H)$)
$$\mathcal{M}\e U=\left( i\hspace{2pt}\mathrm{curl}H,\hspace{2pt} 
-i\hspace{2pt}\mathrm{curl}E\right)$$
subject to the conditions
\begin{gather*}
\mathrm{div}E=\mathrm{div}H=0\text{ in }\mathbf{\Omega}\e,\quad
{E}_\tau=0,\ {H}_\nu=0\text{ on }\mathbf{S}\e.
\end{gather*}
Here ${E}_\tau$ and ${H}_\nu$ are the tangential and normal components of ${E}$ 
and ${H}$, correspondingly. We are interested only on the waves propagated along 
the plane $z=0$, i.e. when $E,H$ depends on $x_1,x_2$ only. 

It is known that if the medium is periodic in two directions and
homogeneous with respect to the third one (so-called $2D$ crystals) then the 
analysis of the Maxwell operator  reduces to the analysis of scalar elliptic 
operators. Let us formulate this statement more precisely. We denote 
\begin{gather*}
J=\big\{(E,H):\ \mathrm{div}E=\mathrm{div}H=0\text{ in }\mathbf{\Omega}\e,\ 
E_\tau=H_\mu=0\text{ on }\mathbf{S}\e\big\},\\
J_E=\{(E,H)\in J:\ E_1=E_2=H_3=0\},\quad J_H=\{(E,H)\in J:\ H_1=H_2=E_3=0\}.
\end{gather*}
The elements of the subspaces $J_E$ and $J_H$ are usually called $E$- and 
$H$-polarized waves. The subspaces $J_E$ and $J_H$ are $L_2$-orthogonal and each 
$U\in J$ can be represented in unique way as
$U=U_E+U_H,$ where $U_E\in J_E,\ U_H\in J_H$. Moreover $J_E$ and $J_H$ are 
invariant subspaces of $\mathcal{M}\e$. Thus $\sigma(\mathcal{M}\e)$ is a union 
of $\sigma(\mathcal{M}\e|_{J_E})$ ($E$-subspectrum) and 
$\sigma(\mathcal{M}\e|_{J_H})$ ($H$-subspectrum).

We denote by $\mathcal{A}_0\e$ and $\mathcal{A}\e$ the Dirichlet and the Neumann 
Laplacians in $\Omega\e$, correspondingly. It can be shown on a formal level of 
rigour (see, e.g, \cite{HMW}) that $\omega\in\sigma(\mathcal{M}\e|_{J_E})$ iff
$\omega^2\in\sigma(\mathcal{A}_0\e)$ and $\omega\in\sigma(\mathcal{M}\e|_{J_H})$ 
iff
$\omega^2\in\sigma(\mathcal{A}\e)$. Using Friedrichs type inequalities one can 
easily prove (see \cite[Lemma 3.1]{Khrab3}) that $(\mathcal{A}_0\e 
u,u)_{L_2(\Omega\e)}\geq a\eps^{-2}\|u\|^2_{L_2(\Omega\e)}$, $\forall 
u\in\mathrm{dom}(\mathcal{A}_0\e)$ (here $a>0$ is a constant) and therefore
 \begin{gather}\label{dir}\min\{\lambda: \lambda\in
\sigma(\mathcal{A}_0\e)\}\underset{{\eps\to 0}}\to \infty.
\end{gather} Then using Theorem \ref{th1}, Lemma \ref{mainlemma} and \eqref{dir} 
we conclude that for an arbitrarily large $L>0$ the Maxwell operator
$\mathcal{M}\e$ has $2m$ gaps in $[-L,L]$ when $\eps$ is small enough and as 
$\eps\to 0$ these gaps converge to intervals 
$\pm(\sqrt{\sigma_j},\sqrt{\mu_j})$, which can be controlled via a suitable 
choice of $B_j$ and $d_j$.

\end{remark}

\section{\label{sec1} Construction of the family $\{\Omega\e\}_{\eps}$ and main 
results}

Let $\eps>0$ be a small parameter and let $n\in
\mathbb{N}\setminus\{1\}$. Let $B_j$, $j=1,\dots m$ be arbitrary
open domains with Lipschitz boundaries satisfying the following
conditions:
\begin{itemize}
\item[$(b_1)$]\ $\overline{B_j}\cap \overline{B_k}=\varnothing$
for $j\not= k$,

\item[$(b_2)$]\ $\cupl_{j=1}^m \overline{B_j}\subset Y$, where
$$Y=\{x=(x_1,\dots,x_n)\in \mathbb{R}^n:\ |x_i|<{1/2}, \forall i\},$$

\item[$(b_3)$]\ for any $j=1,\dots,m$ the boundary of $\partial
B_j$ has a flat subset, namely 
\begin{gather*}
\exists\tilde
x^{j}\in\partial B_j\e,\ \exists r_j>0:\ \text{the set }B_{r_j}(\tilde
x^{j})\cap\partial B_j\text{ belongs to a }(n-1)\text{-dimensional
hyperplain.} 
\end{gather*}
Here by $B_{r}(x)$ we denote the ball with the
center at the point $x$ and the radius $r$.

\end{itemize}
For $j=1,\dots,m$ we denote
\begin{itemize}
\item $D_j\e=\left\{x\in \partial B:\ |x-\tilde x^j|<d_j\e\right\}$, where 
$d_j\e$ is defined by the following formula:
\begin{gather*}
d_j\e=
\begin{cases}
\ds d_j\eps^{2\over n-2},& n>2,\\
\ds \eps^{-1}\exp\left(-{1\over d_j\eps^{2}}\right),& n=2.
\end{cases}
\end{gather*}
Here $d_j$, $j=1,\dots,m$ are positive constants. It is supposed that $\eps$ is
small enough so that $d_j\e<r_j$.

\item $S\e_j=\partial B_j\setminus \left(\cupl_{j=1}^m D_j\e\right)$.

\end{itemize}
Finally we set
\begin{gather*}\label{Omega}
\Omega\e=\mathbb{R}^n\setminus\left(\cupl_{i\in
\mathbb{Z}^n}\cupl_{j=1}^m S_{ij}\e\right),\text{ where
}S_{ij}\e=\eps (S_j\e+i).
\end{gather*}

Let us define precisely the Neumann Laplacian $\mathcal{A}\e$ in $\Omega\e$. We
denote by $\eta\e[u,v]$ the sesquilinear form in
${L_{2}(\Omega\e)}$ which is defined by the formula
\begin{gather}\label{form1}
\eta\e[u,v]=\intl_{\Omega\e}\left(\nabla
u,\nabla \bar v\right)dx
\end{gather}
and the definitional domain $\mathrm{dom}(\eta\e)=H^1(\Omega\e)$.
Here $\left(\nabla u,\nabla
\overline{v}\right)=\ds\suml_{k=1}^n{\partial u\over\partial
x_k}{\partial \overline{v}\over\partial x_k}$. The form
$\eta\e[u,v]$ is densely defined closed and positive. Then (see,
e.g., \cite[Chapter 6, Theorem 2.1]{Kato}) there exists the unique
self-adjoint and positive operator $\mathcal{A}\e$ associated with
the form $\eta\e$, i.e.
\begin{gather}\label{eta-a}
(\mathcal{A}\e u,v)_{L_2(\Omega\e)}= \eta\e[u,v],\quad\forall u\in
\mathrm{dom}(\mathcal{A}\e),\ \forall  v\in \mathrm{dom}(\eta\e).
\end{gather}

We denote by $\sigma(\mathcal{A}\e)$ the spectrum of
$\mathcal{A}\e$. To describe the behaviour of
$\sigma(\mathcal{A}\e)$ as $\eps\to 0$ we need some additional
notations.

In the case $n>2$ we denote by $\kappa$ the capacity of
the disc
$$T=\left\{x=(x_1,\dots,x_n)\in \mathbb{R}^n:\ |x|<1,\ x_n=0\right\}$$
Recall (see, e.g, \cite{Landkof}) that it is defined by
$$\kappa=\inf_w \intl_{\mathbb{R}^n}|\nabla w|^2 dx,$$
where the infimum is taken over smooth and compactly supported in
$\mathbb{R}^n$ functions equal to $1$ on $T$.

We set (below $j=1,\dots,m$)
\begin{gather}\label{sigma}
\sigma_j=\begin{cases}\ds{\kappa d_j^{n-2}\over 4b_j},&n>2,\\
\ds{\pi d_j\over 2b_j},&n=2,
\end{cases}
\end{gather}
where $b_j$ is the volume of the domain $B_j$. We assume that the
numbers $d_j$ and $b_j$ are such that
\begin{gather}
\label{sigma_ord} \sigma_j<\sigma_{j+1},\ j=1,\dots m-1.
\end{gather}

Let us consider the following equation (with unknown
$\lambda\in\mathbb{C}$):
\begin{gather}\label{mu_eq}
1+\suml_{j=1}^m{\sigma_j b_j\over (1-\suml_{i=1}^m
b_i)(\sigma_j-\lambda)}=0.
\end{gather}
It is easy to show (see \cite[Subsect. 3.2]{Khrab1}) that if
\eqref{sigma_ord} holds then equation \eqref{mu_eq} has exactly
$m$ roots, they are real and interlace with $\sigma_j$. We denote
them $\mu_j$, $j=1,\dots,m$ supposing that they are renumbered in
the increasing order, i.e.
\begin{gather}\label{inter}
\sigma_j<\mu_j<\sigma_{j+1},\ j={1,\dots,m-1},\quad
\sigma_m<\mu_m<\infty.
\end{gather}

Now we can formulate the main result on the behaviour of $\sigma(\mathcal{A}\e)$ 
as $\eps\to 0$.
\begin{theorem}\label{th1}
Let $L$ be an arbitrary number satisfying $L>\mu_m$. Then the
spectrum $\sigma(\mathcal{A}\e)$ of the operator $\mathcal{A}\e$
has the following structure in $[0,L]$ when $\eps$ is small
enough:
\begin{gather}\label{th1_f1}
\sigma(\mathcal{A}\e)\cap[0,L]=[0,L]\setminus
\left(\cupl_{j=1}^{m}(\sigma_j\e,\mu_j\e)\right),
\end{gather}
where the intervals $(\sigma_j\e,\mu_j\e)$ satisfy
\begin{gather}\label{th1_f2}
\forall j=1,\dots,m:\quad \lim_{\eps\to
0}\sigma_j\e=\sigma_j,\quad \lim_{\eps\to 0}\mu_j\e=\mu_j.
\end{gather}
\end{theorem}

Theorem \ref{th1} shows that $\sigma(\mathcal{A}\e)$ has exactly $m$ gaps when 
$\eps$ is small enough and when $\eps\to 0$ these gaps converge to the  
intervals $(\sigma_j,\mu_j)$. Now, our goal is to find such numbers $d_j$ and 
domains $B_j$ that the corresponding intervals $(\sigma_j,\mu_j)$ coincide with 
the predefined ones.

We use the notations $d=(d_1,\dots,d_m),\ b=(b_1,\dots,b_m),\ 
\sigma=(\sigma_1,\dots,\sigma_m),\ \mu=(\mu_1,\dots,\mu_m)$. Let 
$$\mathcal{L}:\mathbb{R}^m\times \mathbb{R}^m\to\mathbb{R}^m\times 
\mathbb{R}^m,\quad (d,b)\overset{\mathcal{L}}\mapsto (\sigma,\mu)$$ be the map 
with the definitional domain
$$\mathrm{dom}(\mathcal{L})=\bigg\{(d,b)\in\mathbb{R}^m\times \mathbb{R}^m:\ 
d_j>0,\ b_j>0, \suml_{j=1}^mb_j<1 \text{ and }(\ref{sigma_ord})\text{ 
holds}\bigg\}$$
 and acting according to formulae \eqref{sigma}, \eqref{mu_eq}, \eqref{inter} 
(i.e. $\sigma_j$ are defined by \eqref{sigma} and $\mu_j$ are the roots of 
equation \eqref{mu_eq} renumbered according to \eqref{inter}).

\begin{lemma}\label{mainlemma}
The map ${\mathcal{L}}$ maps $\mathrm{dom}(\mathcal{L})$ onto the set
$$\mathcal{G}=\left\{(\sigma,\mu)\in\mathbb{R}^m\times\mathbb{R}^m:\
\sigma_j<\mu_j<\sigma_{j+1},\ j={1,\dots,m-1},\quad
\sigma_m<\mu_m<\infty\right\}.$$ Moreover $\mathcal{L}$ is
one-to-one and the inverse map $\mathcal{L}^{-1}$ is given by the
following formulae:
\begin{gather}\label{mlm1}
d_j=\begin{cases}\ds\left({4\sigma_j \rho_j\over\kappa\left(1+\suml_{i=1}^m 
\rho_i\right)}\right)^{{1\over n-2}},& n>2,\\
\ds{2\sigma_j \rho_j\over \pi \left(1+\suml_{i=1}^m
\rho_i\right)}, & n=2,
\end{cases}\\\label{mlm2} b_j={\rho_j\over 1+\suml_{i=1}^m
\rho_i},
\end{gather}
where
\begin{gather}\label{rho}
\rho_j={\mu_j-\sigma_j\over\sigma_j}\prod\limits_{i={1,\dots,m}|i\not=
j}\ds\left({\mu_i-\sigma_j\over \sigma_i-\sigma_j}\right).
\end{gather}
\end{lemma}

\begin{proof}
Let  $(\sigma,\mu)$ be an arbitrary element of $\mathcal{G}$. We have to show 
that
$$\exists ! (d,b)\in\mathrm{dom}(\mathcal{L})\text{ such that }\forall j=1,\dots 
m \ \begin{cases}\eqref{sigma}\text{ holds },\\\eqref{mu_eq}\text{ holds with } 
\lambda=\mu_j,\end{cases}$$
moreover, this $(d,b)$ is defined by formulae \eqref{mlm1}-\eqref{rho}.

At first we find $b_1,\dots,b_m$. Let us consider the following system of $m$ 
linear equations with respect to unknowns $\rho_1,\dots,\rho_n$:
\begin{gather}\label{mu_eq_ro}
1+\suml_{j=1}^m{\sigma_j\rho_j\over \sigma_j-\mu_j}=0,\
j=1,\dots,m.
\end{gather}
It is proved in \cite[Lemma 4.1]{Khrab1} that this system has the
unique solution which is defined by formula \eqref{rho}. Therefore
in view of \eqref{mu_eq} in order to find $b_j$ we need to solve
the following system:
\begin{gather*}
{b_j(1-\suml_{i=1}^m b_i)^{-1}}=\rho_j,\ j=1,\dots,m.
\end{gather*}
It is clear that it has the unique solution $b_1,\dots,b_m$ which is defined by 
\eqref{mlm2}.
Since $(\sigma,\mu)\in\mathcal{G}$ then
\begin{gather*}
\forall j:\ \mu_j>\sigma_j;\quad \forall i\not= j:\
\mathrm{sign}(\mu_i-\sigma_j)=\mathrm{sign}(\sigma_i-\sigma_j)\not= 0
\end{gather*}
and hence $\rho_j>0$. Therefore $b_j>0$ and $\suml_{j=1}^m b_j<1$.

Finally knowing $b_j$ we express $d_j$ from \eqref{sigma} and obtain the formula
\eqref{mlm1}. The lemma is proved.
\end{proof}

Now, Theorem \ref{th0} follows \textcolor{black}{directly from} Theorem 
\ref{th1} and Lemma \ref{mainlemma}. Indeed, let $(\mathbf{\a}_j,\b_j)$, 
$j={1,\dots,m}$ be arbitrary intervals satisfying (\ref{intervals}) (and 
therefore by Lemma \ref{mainlemma} $(\a,\b)\in\mathrm{image}(\mathcal{L})$). We 
define the numbers $d_j$, $b_j$ by formulae \eqref{mlm1}-\eqref{mlm2} with 
$\a_j,\b_j$ instead of $\sigma_j$, $\mu_j$. For the obtained numbers $b_j$ we 
construct the domains $B_1,\dots B_m$ satisfying $(b_1)-(b_3)$ and such that
\begin{gather}\label{volume}
|B_j|=b_j\text{ for }j={1,\dots,m}
\end{gather}
(it is easy to do, see example below for one of possible constructions). Finally 
using the domains $B_1,\dots,B_m$ and the numbers $d_1,\dots d_m$ we construct 
the family of periodic domains  $\{\Omega\e\}_\eps$. In view of Theorem 
\ref{th1} the corresponding family of Neumann Laplacians 
$\{\mathcal{A}\e\}_\eps$ satisfies \eqref{spec1}-\eqref{spec2}.

\begin{example}  Let $b_j$, $j=1,\dots,m$ be arbitrary numbers satisfying 
$$b_j>0,\ j=1,\dots,m\text{ and } \suml_{j=1}^m b_j<1.$$
We present one of the possible choices of the domains $B_j$ satisfying 
$(b_1)-(b_3)$ and \eqref{volume}.

We denote:
$$l=\ds\left({1\over 2}+\ds{1\over 2}\suml_{i=1}^m b_i\right)^{1/n},\quad
\hat l=\ds{1\over 2(n-1)l^{n-1}}\left(1-\suml_{i=1}^m b_i\right),\quad
l_j =\ds{b_j\over l^{n-1}}.$$

Finally we define the domains $B_j$, $j=1,\dots,m$ by the following formula:
\begin{gather*}
B_j=\left\{x\in \mathbb{R}^n: x_1\in \left(-{l\over 2}+(j-1)\hat
l+\suml_{i=1}^{j-1}l_j,\ -{l\over 2}+(j-1)\hat
l+\suml_{i=1}^{j}l_j\right),\ |x_k|<{l\over 2},\
k=2,\dots,n\right\}
\end{gather*}
It is easy to show that these domains satisfy conditions $(b_1)-(b_3)$ and 
\eqref{volume}.
\end{example}

\section{\label{sec2}Proof of Theorem \ref{th1}}

\subsection{\label{ss21}Preliminaries}

We present the proof of Theorem \ref{th1} for the case $n\geq 3$
only. For the case $n=2$ the proof is repeated word-by-word with
some small modifications.

In what follows by $C,C_1...$ we denote generic constants that do
not depend on $\eps$.

Let $B$ be an open domain in $\mathbb{R}^n$. By $\langle u
\rangle_B$ we denote the mean value of the function $v(x)$ over
the domain $B$, i.e.
$$\langle u \rangle_B={1\over |B|}\intl_B u(x)dx.$$ Here by
$|B|$ we denote the volume of the domain $B$.

If $\Sigma\subset \mathbb{R}^n$ is a $(n-1)$-dimensional surface
then the Euclidean metrics in $\mathbb{R}^n$ induces on $\Sigma$
the Riemannian metrics and measure. We denote by $ds$ the density
of this measure. Again by $\langle u\rangle_\Sigma$ we denote the
mean value of the function $u$ over $\Sigma$, i.e $\langle
u\rangle_\Sigma=\ds{1\over |\Sigma|}\intl_{\Sigma}u ds$, where
$|\Sigma|=\intl_\Sigma ds$.

We introduce the following sets: $$Y\e=Y\setminus \cupl_{j=1}^m
S_j\e.$$

By $\A\e$ we denote the Neumann Laplacian in $\eps^{-1}\Omega\e$.
It is clear that
\begin{gather}\label{AA}
\sigma(\mathcal{A}\e)=\eps^{-2}\sigma(\A\e).
\end{gather}
It is more convenient to deal with the operator $\A\e$ since the
external boundary of its period cell is fixed (it coincides with
$\partial Y$).

In view of the periodicity of $\A\e$ the analysis of its spectrum
$\sigma(\A\e)$ reduces to the analysis of the spectrum of the
Laplace operator on $Y\e$ with the Neumann boundary conditions on
$\cupl_{j=1}^m S_j\e$ and quasi-periodic boundary (or
$\theta$-periodic) boundary conditions on $\partial Y$. Namely,
let
$$\mathbb{T}^n=\left\{\theta=(\theta_1,\dots,\theta_n)\in
\mathbb{C}^n:\ |\theta_k|=1,\ \forall k\right\}.$$ For $\theta\in
\mathbb{T}^n$ we introduce the functional space $H_\theta^1(Y\e)$
consisting of functions from $H^1(Y\e)$ that satisfy the following
condition on $\partial{Y}$:
\begin{gather}\label{theta1}
\forall k=\overline{1,n}:\quad u(x+ e_k)=\theta_k u(x)\text{ for
}x=\underset{^{\overset{\qquad\qquad\uparrow}{\qquad\qquad
k\text{-th place}}\qquad }}{(x_1,x_2,\dots,-{1/2},\dots,x_n)},
\end{gather}
where $e_k={(0,0,\dots,1,\dots,0)}$.

By $\eta^{\theta,\eps}$ we denote the sesquilenear form defined by
formula (\ref{form1}) (with $Y\e$ instead of $\Omega$) and the
definitional domain $H_\theta^1(Y\e)$. We define the operator
$\A^{\theta,\eps}$ as the operator acting in $L_{2}(Y\e)$ and
associated with the form $\eta^{\theta,\eps}$, i.e.
\begin{gather*}
(\A^{\theta,\eps} u,v)_{L_2(Y\e)}=
\eta^{\theta,\eps}[u,v],\quad\forall u\in
\mathrm{dom}(\A^{\theta,\eps}),\  \forall v\in
\mathrm{dom}(\eta^{\theta,\eps}).
\end{gather*}
The functions from $\mathrm{dom}(\A^{\theta,\eps})$ satisfy the Neumann boundary 
conditions on $\cupl_{j=1}^m S_j\e$,
condition (\ref{theta1}) on $\partial Y$ and the condition
\begin{gather}\label{theta2}
\forall k=\overline{1,n}:\quad {\partial u\over\partial x_k}(x+
e_k)=\theta_k {\partial u\over\partial x_k}(x)\text{ for
}x={(x_1,x_2,\dots,-{1/2},\dots,x_n)}.
\end{gather}

The operator $\A^{\theta,\eps}$ has purely discrete spectrum. We
denote by
$\left\{\lambda_k^{\theta,\eps}\right\}_{k\in\mathbb{N}}$ the
sequence of eigenvalues of $\A^{\theta,\eps}$ written in the
increasing order and repeated according to their multiplicity.

The Floquet-Bloch theory (see, e.g., \cite{Brown,Kuchment,Reed})  establishes
the following relationship between the spectra of the operators
$\A\e$ and $\A^{\theta,\eps}$:
\begin{gather}\label{repres1}
\sigma(\A\e)=\cupl_{k=1}^\infty L_k\text{, where
}L_k=\cupl_{\theta\in \mathbb{T}^n}
\left\{\lambda_k^{\theta,\eps}\right\}.
\end{gather}
The sets $L_k$ are compact intervals.

Also we need the Laplace operators on $Y\e$ with the Neumann
boundary conditions on $\cupl_{j=1}^m S_j\e$ and either Neumann or
Dirichlet boundary conditions on $\partial Y$. Namely, we denote
by $\eta^{N,\eps}$ ({resp.} $\eta^{D,\eps}$) the sesquilinear form
in $L_2(Y\e)$ defined by formula (\ref{form1}) (with $Y\e$ instead
of $\Omega\e$) and the definitional domain $H^1(Y\e)$ ({resp.}
$\widehat{H}^1_0(Y\e)=\left\{u\in H^1(Y\e):\ u=0\text{ on
}\partial Y\right\}$). Then by $\A^{N,\eps}$ ({resp.}
$\A^{D,\eps}$) we denote the operator associated with the form
$\eta^{N,\eps}$ ({resp.} $\eta^{D,\eps}$), i.e.
\begin{gather*}
(\A^{*,\eps} u,v)_{L_2(Y\e)}= \eta^{*,\eps}[u,v],\quad\forall u\in
\mathrm{dom}(\A^{*,\eps}),\ \forall v\in \mathrm{dom}(\eta^{*,\eps}),
\end{gather*}
where $*$ is $N$ ({resp.} $D$).

The spectra of the operators $\A^{N,\eps}$ and $\A^{D,\eps}$ are
purely discrete. We denote by
$\left\{\lambda_k^{N,\eps}\right\}_{k\in\mathbb{N}}$ ({resp.}
$\left\{\lambda_k^{D,\eps}\right\}_{k\in\mathbb{N}}$) the sequence
of eigenvalues of $\A^{N,\eps}$ ({resp.} $\A^{D,\eps}$) written in
the increasing order and repeated according to their multiplicity.

Using the min-max principle (see, e.g., \cite{Reed}) and the
enclosure $H^1(Y\e) \supset H^1_\theta(Y\e)\supset
\widehat{H}^1_0(Y\e)$ one can easily prove the inequality
\begin{gather}\label{enclosure}
\forall k\in \mathbb{N},\ \forall\theta\in \mathbb{T}^n:\quad
\lambda_k^{N,\eps} \leq\lambda_k^{\theta,\eps}\leq
\lambda_k^{D,\eps}.
\end{gather}

\subsection{\label{ss22}Number-by-number convergence of eigenvalues of the 
Dirichlet, Neumann and
$\theta$-periodic Laplacians}

We denote $$B_{m+1}=Y\setminus \cupl_{j=1}^m \overline{B_j\e}.$$
By $\Delta_{B_j}$, $j=1,\dots,m$ we denote the operator which acts
in $L_2(B_j)$ and is defined by the operation $\Delta$ and the
Neumann boundary conditions on $\partial B_j$. By
$\Delta_{B_{m+1}}^N$ ({resp.} $\Delta_{B_{m+1}}^D$,
$\Delta_{B_{m+1}}^\theta$) we denote the operator which acts in
$L_2({B_{m+1}})$ and is defined by the operation $\Delta$, the
Neumann boundary conditions on $\cup_{j=1}^m\partial B_j$ and the
Neumann ({resp.} Dirichlet, $\theta$-periodic) boundary conditions
on $\partial Y$. Finally, we introduce the operators $\A^N$,
$\A^D$, $\A^\theta$ which act in
$\underset{j=\overline{1,m+1}}\oplus L_2(B_j)$ and are defined by
the following formulae:
\begin{gather*}
\A_N=-\left(\begin{matrix}\Delta_{{B_1}}&0&\dots&0\\
0&\Delta_{{B}_2}&\dots&0\\\vdots&\vdots&\ddots&\vdots\\0&0&\dots&\Delta^N_{{B}_{
m+1}}\end{matrix}\right),\
\A_D=-\left(\begin{matrix}\Delta_{{B_1}}&0&\dots&0\\
0&\Delta_{{B}_2}&\dots&0\\\vdots&\vdots&\ddots&\vdots\\0&0&\dots&\Delta^D_{{B}_{
m+1}}\end{matrix}\right),\
\A_\theta=-\left(\begin{matrix}\Delta_{{B_1}}&0&\dots&0\\
0&\Delta_{{B}_2}&\dots&0\\\vdots&\vdots&\ddots&\vdots\\0&0&\dots&\Delta^\theta_{
{B}_{m+1}}\end{matrix}\right),
\end{gather*}
We denote by $\left\{\lambda_k^{N}\right\}_{k\in\mathbb{N}}$
({resp.} $\left\{\lambda_k^{D}\right\}_{k\in\mathbb{N}}$,
$\left\{\lambda_k^{\theta}\right\}_{k\in\mathbb{N}}$) the sequence
of eigenvalues of $\A^{N}$ ({resp. }$\A^{D}$, $\A^{\theta}$)
written in the increasing order and repeated according to their
multiplicity. It is clear that
\begin{gather}\label{lambdaN}
\lambda_1^{N}=\lambda_2^{N}=\dots =\lambda^N_{m+1}=0,\
\lambda_{m+2}^{N}>0,\\\label{lambdaD}
\lambda_1^{D}=\lambda_2^D=\dots =\lambda_m^D=0,\
\lambda_{m+1}^{D}>0,\\\label{lambdaT1}
\lambda_1^{\theta}=\lambda^{\theta}_2=\dots
=\lambda_{m+1}^\theta=0,\ \lambda_{m+2}^{\theta}>0\text{ if
}\theta=(1,1,\dots,1),\\\label{lambdaT2}
\lambda_1^{\theta}=\lambda_2^{\theta}=\dots
=\lambda_m^{\theta}=0,\ \lambda_{m+1}^{\theta}>0\text{ if
}\theta\not=(1,1,\dots,1).
\end{gather}

\begin{theorem}\label{th3}
For each $k\in \mathbb{N}$ one has
\begin{gather}
\label{conv1}\liml_{\eps\to
0}\lambda_k^{N,\eps}=\lambda^N_k,\\\label{conv2}\liml_{\eps\to
0}\lambda_k^{D,\eps}=\lambda^D_k,\\\label{conv3}\liml_{\eps\to
0}\lambda_k^{\theta,\eps}=\lambda^\theta_k\quad
(\forall\theta\in\mathbb{T}^n).
\end{gather}
\end{theorem}

For the case $m=1$ Theorem \ref{th3} was proved in \cite{Khrab3}.
For $m>1$ the proof is similar.

\subsection{\label{ss23}Asymptotics of the first $m$ non-zero eigenvalues of the 
Neumann Laplacian}

We get more complete information about the behaviour of $\lambda_k^{N,\eps}$ , 
$k=2,\dots,m+1$ (it is clear that $\lambda_1^{N,\eps}=0$).

\begin{theorem}\label{Nth}
For $k=2,\dots,m+1$ one has
\begin{gather}\label{Nthmain}
\liml_{\eps\to 0}\eps^{-2}{\lambda}_k^{N,\eps}=\mu_{k-1}.
\end{gather}
\end{theorem}

\begin{proof}

Let $u_k\e$, $k\in\mathbb{N}$ be the eigenfunctions corresponding
to $\lambda_k^{N,\eps}$ and satisfying the conditions
\begin{gather}\label{Nort}
(u_k\e,u_l\e)_{L_2(Y\e)}=\delta_{kl},\\{ u\e_k\text{ are real
functions}}\notag
\end{gather}
(here $\delta_{kl}$ is the Kronecker delta). It is clear that $u_1\e=\pm 1$.

By the min-max principle (see, e.g, \cite{Reed})  we get
\begin{gather}\label{Nminimax1}
\forall k\in \mathbb{N}:\ \lambda_k^{N,\eps}=\minl_{u\in
H(u_1\e,\dots,u\e_{k-1}) }{\|\nabla u\|^2_{L_2(Y\e)}\over
\|u\|^2_{L_2(Y\e)}},
\end{gather}
where
\begin{gather}\label{hk}
H(u\e_1,\dots,u\e_{k-1}) =\{u\in H^1(Y\e):\ (u,u_l\e)_{L_2(Y\e)},\  
l=1,\dots,k-1\}.
\end{gather}

Using the Cauchy inequality we get the estimate
$$
\suml_{j=1}^{m+1}\left(\langle u_{k}\e\rangle_{B_j\e}\right)^2\leq
C
$$
and therefore there exist a subsequence (for convenience still denoted by 
$\eps$) and numbers $\mathbf{s}^k_{j}\in\mathbb{R}$, $k\in \mathbb{N}$, 
$j=1,\dots,m+1$ such that
\begin{gather}\label{Nbetadef}
\liml_{\eps\to 0}\langle u_{k}\e\rangle_{B_j}=\mathbf{s}^k_{j}.
\end{gather}
We denote 
$\mathbf{s}^k=(\mathbf{s}^k_{1},\dots,\mathbf{s}^k_{m+1})\in\mathbb{R}^{m+1}$.

During the proof we will use the function $F:\mathbb{R}^{m+1}\to
[0,\infty)$ defined by the formula (below
$s=(s_1,\dots,s_{m+1})\in\mathbb{R}^{m+1}$)
\begin{gather}\label{F}
F({s})=\suml_{j=1}^m\sigma_{j}b_j(s_{m+1}-s_{j})^2.
\end{gather}

Also we will use the function $U\e:Y\e\times \mathbb{R}^{m+1}\to\mathbb{R}$ 
which is defined by the following formula (below $x\in Y\e$, 
$s=(s_1,\dots,s_{m+1})\in \mathbb{R}^{m+1}$):
\begin{gather}\label{Nvb}
U\e(x,s)=
\begin{cases}\ds
s_j-{s_j-s_{m+1}\over
2}\psi\left(\mathcal{R}^{-1}_j\left({x-\tilde x^j\over
d_j\e}\right)\right)\varphi\left({|x-\tilde x^j|\over r}\right)&\text{if }x\in 
B_j,\ j=1,\dots m,\\
\ds s_{m+1}+\suml_{i=1}^m{s_i-s_{m+1}\over 
2}\psi\left(\mathcal{R}_j^{-1}\left(x-\tilde
x^i\over d_i\e\right)\right)\varphi\left({|x-\tilde x^i|\over
r}\right)&\text{if }x\in B_{m+1}.
\end{cases}
\end{gather}
Here $\varphi:\mathbb{R}\to \mathbb{R}$ is a twice-continuously
differentiable function satisfying
\begin{gather}\label{varphi}
\varphi(\rho)=1\text{ as }\rho\leq 1/4\text{ and
}\varphi(\rho)=0\text{ as }\rho\geq 1/2,
\end{gather}
$r$ is an arbitrary positive constant satisfying the conditions
\begin{gather}\label{r}
r<\minl_{j=\overline{1,m}}r_j;\quad B_r(\tilde x^j)\cap
\left(\partial Y\cupl_{i\not=j} B_i\right)=\varnothing;
\end{gather}
$\mathcal{R}_j:\mathbb{R}^{n}\to\mathbb{R}^{n}$ ($j=1,\dots,m$) is
the operator of rotation mapping the disc $T$ onto a set which is parallel
to a flat part of $\partial B_j$ containing $\tilde x^i$; $\psi$
is a solution of the following problem:
\begin{gather}\label{bvp-}
\Delta\psi=0\text{ in }\mathbb{R}^n\setminus
\overline{T},\\\label{bvp} \psi=1\text{ in }\partial
T,\\\label{bvp+} \psi(x)=o(1)\text{ as }|x|\to\infty
\end{gather}
(recall that $T=\left\{x\in \mathbb{R}^n:\ |x|<1,\ x_n=0\right\}$,
obviously $\partial T=\overline{T}$).  It is well-known that problem 
\eqref{bvp-}-\eqref{bvp+} has the unique solution $\psi(x)$ satisfying 
$\intl_{\mathbb{R}^n\setminus T}|\nabla\psi|^2 dx<\infty$.

Using a standard regularity theory it is easy to prove that
$\psi(x)$ has the following properties:
\begin{gather}\label{psi1}
\psi\in C^\infty(\mathbb{R}^n\setminus\overline{T}),\\\label{psi2}
\psi(x_1,x_2,\dots,x_{n-1},x_n)=\psi(x_1,x_2,\dots,x_{n-1},-x_n).
\end{gather}
These properties imply:
\begin{gather}\label{psi4}
{\partial\psi\over\partial x_n}=0\text{ in } \left\{x\in
\mathbb{R}^n:\ x_n=0\right\}\setminus \overline T,\\\label{psi5}
\left.{\partial\psi\over\partial
x_n}\right|_{x_n=+0}+\left.{\partial\psi\over\partial
x_n}\right|_{x_n=-0}=0\text{ in } T.
\end{gather}
Furthermore the function $\psi(x)$ satisfies the estimate (see,
e.g, \cite[Lemma 2.4]{March}):
\begin{gather}\label{cap_est}
|D^\a\psi(x)|\leq {C|x|^{2-n-\a}}\ \text{ for }|x|>2,\ |\a|=0,1,2.
\end{gather}
And finally one has the following equality:
\begin{gather}
\label{psi3}\kappa=\ds\intl_{\mathbb{R}^n\setminus
T}|\nabla\psi|^2 dx.
\end{gather}

It is easy to see that for $\eps$ small enough (namely, when
$\maxl_j d_j\e< r/4$) and for any $s\in\mathbb{R}^{m+1}$
$U\e(\cdot,s)\in \mathrm{dom}(\A^{N,\eps})$ due to
(\ref{varphi}), (\ref{r}), (\ref{bvp}), (\ref{psi1}),
(\ref{psi4}), (\ref{psi5}).

Let us establish some properties of  $U\e(x,s)$ for fixed
$s\in\mathbb{R}^{m+1}$. Using (\ref{cap_est}),
(\ref{psi3})  we obtain:
\begin{gather}\label{Nvbest1}
\|\nabla U\e(\cdot,s)\|^2_{L_2(Y\e)}\sim {1\over 4}\kappa
d^{n-2}\eps^2\suml_{j=1}^{m}(s_{m+1}-s_j)^2=\eps^2 F(s)\quad
(\eps\to 0), \\
\label{Nvbest2}\|\Delta U\e(\cdot,s)\|^2_{L_2(Y\e)}\leq C\eps^4.
\end{gather}
Since $\big[U\e(\cdot,s)-s_j\big]_{int}=0$ on $\partial
B\setminus\left\{x:\ |x-\tilde x_j|\leq r/2\right\}$,
$j=1,\dots,m$  (here $[\dots]_{int}$ denotes the value of the
function when we approach $\partial B_j$ from inside of $B_j$) and
$U\e(\cdot,s)-s_{m+1}=0$ on $\partial Y$ we have the following
Friedrichs inequalities
\begin{gather*}
\|U\e(\cdot,s)-s_j\|^2_{L_2(B_j)}\leq C\|\nabla
U\e(\cdot,s)\|^2_{L_2(B_j)}\leq C_1\eps^2,\ j=1,\dots,m+1
\end{gather*}
and therefore
\begin{gather}\label{Nvbest3}
\forall j=1,\dots,m+1:\ U\e(\cdot,s)\underset{\eps\to 0}\to
s_j\text{ strongly in }L_2(B_j).
\end{gather}

By $(\cdot,\cdot)_B$ we denote the following scalar product in
$\mathbb{R}^{m+1}$:
\begin{gather*}
(s,t)_B=\suml_{j=1}^{m+1}s_j t_j|B_j|,\ s,t\in \mathbb{R}^{m+1}.
\end{gather*}
It follows from \eqref{Nvbest3} that
\begin{gather}\label{Nvbest3+}
\|U\e(\cdot,s)\|^2_{L_2(Y\e)}\sim (s,s)_B\text{ as }\eps\to 0.
\end{gather}

\begin{lemma}\label{Nlm1}
One has for $k=1,\dots,m+1$:
\begin{gather}\label{Nlm1main}
\liml_{\eps\to 0}\lambda_k^{N,\eps}\leq C\eps^2.
\end{gather}
\end{lemma}

\begin{proof}
We prove this lemma by induction. For $k=1$ \eqref{Nlm1main} is
obvious (namely, $\lambda_1^{N,\eps}=0$). Now, let
\begin{gather}\label{Nlm1main+}
\liml_{\eps\to 0}\lambda_k^{N,\eps}\leq C\eps^2\text{ for }k\leq k'-1
\end{gather}
and let us prove \eqref{Nlm1main+} for $k=k'$.

One has the following Poincar\'{e} inequalities:
\begin{gather*}
\|u_k\e-\langle u_{k}\e\rangle_{B_j}\|^2_{L_2(B_j)}\leq C\|\nabla
u_k\e\|^2_{L_2(Y\e)}= C\lambda_k^{N,\eps},\ j=1,\dots,m+1,\ k=1,2,3\dots
\end{gather*}
and therefore due to \eqref{Nbetadef}, \eqref{Nlm1main+} one has
\begin{gather}\label{Nconv1}
\forall l=1,\dots,k'-1,\ \forall j=1,\dots,m+1:\
u_{l}\e\underset{\eps\to 0}\to \mathbf{s}^{l}_j\text{ strongly in
}L_2(B_j).
\end{gather}

Now, let $\hat s\in\mathbb{R}^{m+1}\setminus\{0\}$ be an arbitrary vector 
satisfying:
\begin{gather}\label{hats}
(\hat{s},\mathbf{s}^l)_B=0,\ l=1,\dots,k'-1.
\end{gather}
The choice of such a vector is always possible whenever $k'\leq m+1$. We denote
$$\hat u\e(x)=U\e(x,\hat s)-w\e(x,\hat s)$$
where
\begin{gather*}
w\e(x,\hat s)=\suml_{l=1}^{k'-1}(U(\cdot,\hat
s),u_l\e(\cdot))_{L_2(Y\e)}u_l\e(x).
\end{gather*}
In view of \eqref{Nvbest3}, \eqref{Nconv1}, \eqref{hats} we obtain
\begin{gather}\label{ort}
(U\e(\cdot, \hat s),u_l\e(\cdot))_{L_2(Y\e)}\underset{\eps\to
0}\to (\hat{s},\mathbf{s}^l)_B=0.
\end{gather}
Using  \eqref{Nlm1main+} and \eqref{ort} one has
\begin{gather}\label{w1}
\|w\e(\cdot,\hat s)\|_{L_2(Y\e)}^2=\suml_{l=1}^{k'-1}(U\e(\cdot,\hat
s),u_l\e(\cdot))^2_{L_2(Y\e)}\underset{\eps\to 0}\to
0,\\\label{w2} \eps^{-2}   \|\nabla w\e(\cdot,\hat
s)\|_{L_2(Y\e)}^2=\eps^{-2}\suml_{l=1}^{k'-1}\lambda_l^{N,\eps}(U\e(\cdot,\hat
s),u_l\e(\cdot))^2_{L_2(Y\e)}\leq
C\suml_{l=1}^{k'-1}(U\e(\cdot,\hat
s),u_l\e(\cdot))^2_{L_2(Y\e)}\underset{\eps\to 0}\to 0.
\end{gather}

Obviously $\hat u\e\in H(u_1\e,\dots,u_{k'-1}\e)$. Then using
\eqref{Nminimax1}, \eqref{Nvbest1},
\eqref{Nvbest3+}, \eqref{w1}, \eqref{w2} we obtain
\begin{gather*}
\lambda_{k'}^{N,\eps}\leq {\|\nabla \hat u\e\|^2_{L_2(Y\e)}\over
\| \hat u\e\|^2_{L_2(Y\e)}}\sim {\|\nabla {U}\e(\cdot,\hat
s)\|^2_{L_2(Y\e)}\over  \| {U}\e(\cdot,\hat s)\|^2_{L_2(Y\e)}}\sim
{F(\hat s)\eps^2\over (\hat s,\hat s)_B}\leq C\eps^2.
\end{gather*}
The lemma is proved.

\end{proof}

\begin{lemma}\label{Nlm2}
One has for $k=1,\dots,m+1$:
\begin{gather}\label{Nlm2main}
\liml_{\eps\to 0}\eps^{-2}\lambda_k^{N,\eps}=F(\mathbf{s}^k).
\end{gather}
\end{lemma}

\begin{proof}

Using the \textcolor{black}{Poincar\'{e}} inequality and Lemma \ref{Nlm1}  we 
obtain the following estimates:
\begin{gather*}
\|u_k\e-\langle u_{k}\e\rangle_{B_j}\|^2_{L_2(B_j)}\leq C\eps^2,\ j,k=1,\dots, 
m+1
\end{gather*}
and therefore in view of \eqref{Nbetadef}
\begin{gather}\label{Nconv2}
\forall k=1,\dots,m+1,\ \forall j=1,\dots,m+1:\
u_{k}\e(\cdot,s)\underset{\eps\to 0}\to \mathbf{s}^{k}_j\text{
strongly in }L_2(B_j).
\end{gather}
Using \eqref{Nconv2} we get from \eqref{Nort}
\begin{gather}\label{Nconv3}
(\mathbf{s}^k,\mathbf{s}^l)=\delta_{kl}.
\end{gather}

We construct an approximation $\mathbf{u}_k^\eps\in H(u_1\e,\dots,u_{k-1}\e)$ 
for the
eigenfunction $u_k\e$ by the formula
\begin{gather}\label{uappr}
\mathbf{u}_k\e(x)=U\e(x,\mathbf{s}^k)-\mathbf{w}\e_k(x),
\end{gather}
where
$$\mathbf{w}_k\e(x)=
\suml_{l=1}^{k-1}({U}\e(\cdot,\mathbf{s}^k),u\e_l(\cdot))_{L_2(Y\e)}u_l\e(x).$$

Using \eqref{Nvbest1},  \eqref{Nvbest2} \eqref{Nvbest3+} and \eqref{Nconv3}  we 
obtain
\begin{gather}\label{Uprop}
\|\nabla U\e(\cdot,\mathbf{s}^k)\|^2_{L_2(Y\e)}\sim\eps^2
F(\mathbf{s}^k),\quad \|U\e(\cdot,\mathbf{s}^k)\|^2_{L_2(Y\e)}\sim
(\mathbf{s}^k,\mathbf{s}^k)_B=1,\quad \|\Delta
U\e(\cdot,\mathbf{s}^k)\|^2_{L_2(Y\e)}\leq C\eps^4.
\end{gather}

The functions $\mathbf{w}\e_k(x)$ brings vanishingly small
contribution to $\mathbf{u}_k\e$. Namely, using \eqref{Nvbest3},
\eqref{Nlm1main}, \eqref{Nconv2}, \eqref{Nconv3} we get
\begin{multline}\label{wprop}
\|\mathbf{w}_k\e\|^2_{L_2(Y\e)}+\eps^{-2}\|\nabla\mathbf{w}_k\e\|^2_{L_2(Y\e)}
+\eps^{-4}\|\Delta\mathbf{w}_k\e\|^2_{L_2(Y\e)}\leq\\\leq
C\suml_{l=1}^{k-1}(U(\cdot,{
\mathbf{s}}^k),u_l\e(\cdot))_{L_2(Y\e)}^2\underset{\eps\to
0}\to
C\suml_{l=1}^{k-1}(\mathbf{s}^k,\mathbf{s}^l)^2=0.
\end{multline}

Now let us estimate the difference
\begin{gather}\label{Ndifference}
\delta_k\e=u_k\e-\mathbf{u}_k\e.
\end{gather}
Taking into account \eqref{Nvbest3}, \eqref{Nconv2}, \eqref{wprop} we get
\begin{gather}\label{Ndelta1}
\|\delta_k\e\|^2_{L_2(Y\e)}\leq
3\left(\suml_{j=1}^{m+1}\|u_k\e-\mathbf{s}^k_j\|^2_{L_2(B_j\e)}+\suml_{j=1}^{m+1
}\|\mathbf{s}^k_j-{U}\e(\cdot,\mathbf{s}^k)\|^2_{L_2(B_j\e)}+\|\mathbf{w}
_k\e\|^2_{L_2(Y\e)}\right)\underset{\eps\to
0}\to 0.
\end{gather}
Since $\mathbf{u}_k\e\in H(u_1\e,\dots,u_{k-1}\e)$ we get
\begin{gather}\label{Nminimax}
\|\nabla u_k\e\|^2_{L_2(Y\e)}\leq {\|\nabla
\mathbf{u}_k\e\|^2_{L_2(Y\e)}\over \|
\mathbf{u}_k\e\|^2_{L_2(Y\e)}}.
\end{gather}
Plugging \eqref{Ndifference}  into (\ref{Nminimax}) and
integrating by parts we obtain
\begin{gather*}\label{Ndeltaprem}
\|\nabla \delta_k\e\|^2_{L_2(Y\e)}\leq 2|(\Delta
\mathbf{u}_k\e,\delta_k\e)_{L_2(Y\e)}|+\|\nabla
\mathbf{u}_k\e\|^2_{L_2(Y\e)}\left(\|\mathbf{u}_k\e\|_{L_{2}
(Y\e)}^{-2}-1\right).
\end{gather*}
and then in view of \eqref{Uprop}, \eqref{wprop}, \eqref{Ndelta1} we conclude
that
\begin{gather}\label{Ndelta2}
\liml_{\eps\to 0}\eps^{-2}\|\nabla\delta_k\e\|^2_{L_2(Y\e)}=0.
\end{gather}

Finally using \eqref{Uprop}, \eqref{wprop}, \eqref{Ndelta2} we obtain
\begin{gather}\label{nlast}
\eps^{-2}\lambda^{N,\eps}_{k}= \eps^{-2}{\|\nabla
{u}_k\e\|^2_{L_2(Y\e)}}\sim \eps^{-2}{\|\nabla
\mathbf{u}_k\e\|^2_{L_2(Y\e)}}\sim \eps^{-2}{\|\nabla
{U}\e(\cdot,\mathbf{s}^k)\|^2_{L_2(Y\e)}}\sim F(\mathbf{s}^k) \quad (\eps\to 0).
\end{gather}
The lemma is proved.

\end{proof}

\begin{lemma}\label{Nlm3}

The vectors $\mathbf{s}^{k}$, $k=1,\dots,m+1$ satisfy the
following inequalities:
\begin{gather*}
F(\mathbf{s}^{k})\leq F(s)\text{ for any }s\in
\mathbf{H}({\mathbf{s}}^1,\dots,{\mathbf{s}}^{k-1}),
\end{gather*}
where
$$\mathbf{H}({\mathbf{s}}^1,\dots,{\mathbf{s}}^{k-1})=\{s\in\mathbb{R}^{m+1}:\ 
(s,s)_B=1,\ (s,{\mathbf{s}}^l)_B=0,\ l=\overline{1,k-1}\}.$$

\end{lemma}

\begin{proof}
Let $\tilde{\mathbf{s}}^{k}\in \mathbf{H}(\mathbf{s}^1,\dots,\mathbf{s}^{k-1})$ 
be an arbitrary vector satisfying
$$F(\tilde{\mathbf{s}}^{k})\leq F(s)\text{ for any }s\in 
\mathbf{H}(\mathbf{s}^1,\dots,\mathbf{s}^{k-1})$$
(i.e. $\tilde{\mathbf{s}}^{k}$ is a minimizer of $F(s)$ on 
$\mathbf{H}({\mathbf{s}}^1,\dots,{\mathbf{s}}^{k-1})$). Since 
${\mathbf{s}}^{k}\in \mathbf{H}({\mathbf{s}}^1,\dots,{\mathbf{s}}^{k-1})$ then
\begin{gather}\label{ff1}
F(\tilde{\mathbf{s}}^{k})\leq F({\mathbf{s}}^{k}).
\end{gather}

Using the min-max principle we get the inequality
\begin{gather}\label{ff3}
\eps^{-2}\lambda_k^{N,\eps}\leq {\eps^{-2}\|\nabla
{U}\e(\cdot,\tilde{\mathbf{s}}^{k}) -\nabla
\tilde{w}_k\e\|^2_{L_2(Y\e)}\over \|
{U}\e(\cdot,\tilde{\mathbf{s}}^{k})-\tilde{w}_k\e\|^2_{L_2(Y\e)}},
\end{gather}
where $\tilde{w}_{k}\e(x)=
\ds\suml_{l=1}^{k-1}\left({U}\e_k(\cdot,\tilde{\mathbf{s}}^{k}),
u\e_l(\cdot)\right)_{L_2(Y\e)}u_l\e(x)$.
Using the same arguments as in Lemmas \ref{Nlm1}, \ref{Nlm2} one can easily 
prove that
\begin{gather*}
\eps^{-2}\|\nabla {U}\e(\cdot,\tilde{\mathbf{s}}^{k})\|^2_{L_2(Y\e)}\sim 
F(\tilde{\mathbf{s}}^{k}),\quad \| 
{U}\e(\cdot,\tilde{\mathbf{s}}^{k})\|^2_{L_2(Y\e)}\sim 
(\tilde{\mathbf{s}}^{k},\tilde{\mathbf{s}}^{k})_B=1,\\
\liml_{\eps\to 0}\left(\eps^{-2}\|\nabla 
\tilde{w}_{k}\e\|^2_{L_2(Y\e)}+\|\tilde{w}_{k}\e\|^2_{L_2(Y\e)}\right) =0
\end{gather*}
and therefore \eqref{ff3} implies
\begin{gather}\label{ff4}
\liml_{\eps\to 0}\eps^{-2}\lambda_k^{N,\eps}\leq
F(\tilde{\mathbf{s}}^{k}).
\end{gather}

It follows from \eqref{Nlm2main}, \eqref{ff1}, \eqref{ff4} that
\begin{gather*}
\liml_{\eps\to
0}\eps^{-2}\lambda_{k}^{N,\eps}=F(\tilde{\mathbf{s}}^{k}).
\end{gather*}
The lemma is proved.
\end{proof}

It is more convenient to work with the usual scalar product in
$\mathbb{R}^{m+1}$ (instead of the product $(\cdot,\cdot)_B$). In
this connection we reformulate Lemmas \ref{Nlm2}-\ref{Nlm3}.
We introduce the function
$\mathcal{F}:\mathbb{R}^{m+1}\to[0,\infty)$ by
$$\mathcal{F}(q)=\suml_{j=1}^m \sigma_j b_{j}\left({q_{m+1}\over 
\sqrt{|B_{m+1}|}}-
{q_{j}\over \sqrt{|B_{j}|}}\right)^2.$$ We also introduce the
vectors $\mathbf{q}^k\in \mathbb{R}^{m+1}$, $k=1,\dots,m+1$ by the
formula
$$\mathbf{q}_j^k=\mathbf{s}_j^k\sqrt{|B_j|},\ j=1,\dots,m+1.$$

Then, obviously,  Lemmas \ref{Nlm2}-\ref{Nlm3} can be reformulated in the 
following way:

\begin{corollary}\label{Ncor}
One has for $k=1,\dots,m+1$:
\begin{gather}\label{Nlmadd}
\liml_{\eps\to
0}\eps^{-2}\lambda_{k}^{N,\eps}=\mathcal{F}(\mathbf{q}^{k})\leq
\mathcal{F}(q)\text{ for any }q\in
\mathcal{H}({\mathbf{q}}^1,\dots,{\mathbf{q}}^{k-1}),
\end{gather}
where
$$\mathcal{H}({\mathbf{q}}^1,\dots,{\mathbf{q}}^{k-1})=\{q\in\mathbb{R}^{m+1}:\ 
(q,q)=1,\ (q,{\mathbf{q}}^l)=0,\ l=\overline{1,k-1}\}.$$
Here $(\cdot , \cdot)$ is the usual scalar product in
$\mathbb{R}^{m+1}$, i.e. $(q,s)=\suml_{j=1}^{m+1}q_js_j$.
\end{corollary}

It is clear that $\mathbf{q}^{1}$ is either
$\left(\sqrt{|B_1|},\sqrt{|B_2|},\dots,\sqrt{|B_{m+1}|}\right)$ or
$-\left(\sqrt{|B_1|},\sqrt{|B_2|},\dots,\sqrt{|B_{m+1}|}\right)$.

Let us denote by  $E_\varkappa$ (here $\varkappa>0$ is a parameter) the
$(m-1)$-dimensional ellipsoid which is a cross-section of the
elliptic cylinder $\mathcal{F}(q)=\varkappa$ by the hyperplane
$\{q\in\mathbb{R}^{m+1}:\ (q,{\mathbf{q}}^1)=0\}$:
\begin{gather*}
E_\varkappa=\{q\in\mathbb{R}^{m+1}:\ \mathcal{F}(q)=\varkappa,\
(q,{\mathbf{q}}^1)=0\}.
\end{gather*}
We denote by $h_1(\varkappa)\geq h_2(\varkappa)\geq\dots\geq h_m(\varkappa)$ the
half-axes of this ellipsoid (recall that there is some orthogonal
change of variables $q\mapsto\tilde q$ such that $E_\varkappa$ has the
following form in coordinates $x$:
$$\tilde q_{m+1}=0,\ \suml_{j=1}^{m} \left({\tilde q_j\over 
h_j(\varkappa)}\right)^2=1\quad\text{ ).}$$

By $S$ we denote the $(m-1)$-dimensional unit sphere which is a
cross-section of the $m$-dimensional sphere
$\{q\in\mathbb{R}^{m+1}:\ (q,q)=1\}$ by the plane
$\{q\in\mathbb{R}^{m+1}:\ (q,{\mathbf{q}}^1)=0\}$.

Let $\varkappa_1, \varkappa_2,\dots \varkappa_{m}$ be the numbers satisfying
$$h_k(\varkappa_k)=1.$$
We also set $\varkappa_0=0$. It is clear that $\varkappa_j\leq\varkappa_{j+1}$. 
Later
we will prove (see the end of the proof of Lemma \ref{Nlm5}) that
if $\sigma_j<\sigma_{j+1}$ ($j=1,\dots,m-1$) then
\begin{gather}\label{2po}
\varkappa_j<\varkappa_{j+1},\ j=0,\dots,m-1.
\end{gather}
The ellipsoids $E_{\varkappa_k}$, $k=1,\dots,m$ touch the sphere $S$.
Taking into account \eqref{2po} it is easy to show that they touch
$S$ only in two points $\tilde{\mathbf{q}}^k_{-}$ and
$\tilde{\mathbf{q}}^k_{+}$ which are symmetric to each other with
respect to the origin and satisfies the following properties:
\begin{gather}\label{t1}
\forall k,l=\overline{1,m}:\
(\tilde{\mathbf{q}}^k_{\pm},\tilde{\mathbf{q}}^l_{\pm})=\delta_{kl},\\\label{t2}
\forall k=\overline{1,m},\ \forall \varkappa\in 
(\varkappa_{k-1},\varkappa_{k}):\
E_\varkappa\cap \left\{q\in\mathbb{R}^{m+1}:\ (q,q)=1,\
(q,\mathbf{q}^l)=0,\
l=\overline{1,k-1}\right\}=\varnothing,\\\label{t3} \forall
k=\overline{1,m}:\
\mathcal{F}(\tilde{\mathbf{q}}^k_\pm)=\varkappa_{k}.
\end{gather}

The following lemma follows easily from Corollary \ref{Ncor} and 
\eqref{2po}-\eqref{t3}.
\begin{lemma}\label{Nlm4}
One has for $k=2,\dots,m+1$
\begin{gather*}
\mathbf{q}^k=\tilde{\mathbf{q}}^{k-1}_{-}\text{ or
}\mathbf{q}^k=\tilde{\mathbf{q}}^{k-1}_{+} \text{(and therefore
}\liml_{\eps\to
0}\eps^{-2}\lambda_{k}^{N,\eps}=\varkappa_{k-1}\text{)}.
\end{gather*}
\end{lemma}

Now, let us make a change of variables $q\overset{f}\mapsto p$:
\begin{gather*}
p_j=\sqrt{\sigma_j|B_{j}|}\left({q_{m+1}\over \sqrt{|B_{m+1}|}}-
{q_{j}\over \sqrt{|B_{j}|}}\right),\ j=1,\dots,m,\quad
p_{m+1}=\suml_{i=1}^{m+1}q_i\sqrt{|B_i|}.
\end{gather*}
Simple calculations shows that $f$ maps
\begin{itemize}
\item the plane $\{q\in\mathbb{R}^{m+1}:\ (q,{\mathbf{q}}_1)=0\}$
onto the plane $\{p\in\mathbb{R}^{m+1}:\ p_{m+1}=0\}$,

\item the ellipsoid $E_\varkappa$ onto the sphere
$$S_\varkappa=\{p\in\mathbb{R}^{m+1}:\ \suml_{i=1}^m p_i^2=\varkappa,\ 
p_{m+1}=0\},$$

\item the sphere $S$ onto the ellipsoid
$$E=\left\{p\in\mathbb{R}^{m+1}:\ \suml_{i=\overline{1,m}}
p_i^2{1-|B_i|\over\sigma_i}-\suml_{i,j=\overline{1,m};i\not=j} p_i
p_j\sqrt{|B_i||B_j|\over\sigma_i\sigma_j} =1,\
p_{m+1}=0\right\}.$$

\end{itemize}

As any linear non-degenerate map $f$ preserves tangency points,
i.e. $S_{\varkappa_k}$ touches $E$ in the points
$\tilde{\mathbf{p}}^{k}_{\pm}=f(\tilde{\mathbf{q}}^{k}_{\pm})$. We
denote $h_1\leq h_2\le\dots\leq h_m$ the semiaxes of the ellipsoid
$E$. It is clear that $S_\varkappa$ touches $E$ iff  for some $j$ one has
$\varkappa=h_j^2$. Therefore using Lemma \ref{Nlm4} we get for $k=2,\dots
m+1$
\begin{gather*}
\liml_{\eps\to 0}\eps^{-2}\lambda_{k}^{N,\eps}=h_{k-1}^2.
\end{gather*}

Thus in order to complete the proof of Theorem \ref{Nth} it
remains to prove the following lemma.

\begin{lemma}\label{Nlm5} One has for $k=1,\dots,m$:
\begin{gather}\label{hmu}
\mu_k=h_k^2.
\end{gather}
\end{lemma}

\begin{proof}
It is well-known that the numbers $h_k^{2}$  are the roots of the equation
\begin{gather}\label{det}
\mathrm{det}(M-\lambda^{-1} I)=0,\ \lambda\in\mathbb{C}\text{ is
unknown number},
\end{gather}
where $I$ is the identity $m\times m$ matrix, and the matrix 
$M=\{M_{ij}\}_{i,j=1}^m$ is defined by
$$M_{ij}=\ds{1-|B_i|\over\sigma_i}\text{\ if }i=j\text{\quad and\quad 
}M_{ij}=-\ds\sqrt{|B_i||B_j|\over\sigma_i\sigma_j}\text{\ if }i\not= j.$$

We denote by $M(i_1,i_2,\dots,i_k)$ the minor of the matrix
$M$ which is on the intersection of $i_1$-th,
$i_2$-th,$\dots,i_k$-th rows
 and the columns with the same indexes. One has
the following formula (see \cite{Marcus}):
\begin{gather}\label{MM}
\mathrm{det}(M-\lambda^{-1} I)=\suml_{k=0}^m\lambda^{k-m}E_k(M),
\end{gather}
where 
\begin{gather}\label{MM+}
E_0=(-1)^{m},\quad E_k(M)=(-1)^{m-k}\suml_{1\leq
i_1<i_2<\dots < i_k\leq m}M(i_1,i_2,\dots,i_k).
\end{gather}

Let us prove that for $k=1,\dots,m$
\begin{gather}\label{minor}
M(i_1,i_2,\dots,i_k)={1-|B_{i_1}|-|B_{i_2}|-\dots-|B_{i_k}|\over 
\sigma_{i_1}\sigma_{i_2}\dots\sigma_{i_k}}.
\end{gather}
We carry out the proof by induction. For $k=1$ and $k=2$
\eqref{minor} can be easily proved via direct calculations. Now,
suppose that (\ref{minor}) is valid for $k=l-1,l-2$ and let us
prove it for $k=l$. Obviously it is enough to prove (\ref{minor})
only for $i_1=1,i_2=2,\dots,i_l=l$. One has
\begin{multline*}
M(1,2,\dots,l)=\mathrm{det}\left(
\begin{matrix}
{1-B_1\over\sigma_1}&-\sqrt{|B_1||B_2|\over\sigma_1\sigma_2}&\dots 
&-\sqrt{|B_1||B_l|\over\sigma_1\sigma_l}\\-\sqrt{|B_2||B_1|\over\sigma_2\sigma_1
} & {1-B_2\over\sigma_2} & \dots &
-\sqrt{|B_2||B_l|\over\sigma_2\sigma_l}\\\vdots&\vdots&\ddots&\vdots\\-\sqrt{
|B_l||B_1|\over\sigma_l\sigma_1}& -\sqrt{|B_l||B_2|\over\sigma_l\sigma_2}&\dots 
&{1-B_l\over\sigma_l}
\end{matrix}\right)= {1\over\sigma_1}M(2,3,\dots,l)+\\+\mathrm{det}\left(
\begin{matrix}
-{B_1\over\sigma_1}&-\sqrt{|B_1||B_2|\over\sigma_1\sigma_2}&\dots 
&-\sqrt{|B_1||B_l|\over\sigma_1\sigma_l}\\0& {1\over\sigma_2} & \dots &
0\\\vdots&\vdots&\ddots&\vdots\\-\sqrt{|B_l||B_1|\over\sigma_l\sigma_1}& 
-\sqrt{|B_l||B_2|\over\sigma_n\sigma_2}&\dots &{1-B_l\over\sigma_l}
\end{matrix}\right)+\mathrm{det}\left(
\begin{matrix}
-{B_1\over\sigma_1}&-\sqrt{|B_1||B_2|\over\sigma_1\sigma_2}&\dots 
&-\sqrt{|B_1||B_l|\over\sigma_1\sigma_l}\\-\sqrt{|B_2||B_1|\over\sigma_2\sigma_1
} & -{B_2\over\sigma_2} & \dots &
-\sqrt{|B_2||B_l|\over\sigma_2\sigma_l}\\\vdots&\vdots&\ddots&\vdots\\-\sqrt{
|B_l||B_1|\over\sigma_l\sigma_1}& -\sqrt{|B_l||B_2|\over\sigma_l\sigma_2}&\dots 
&{1-B_l\over\sigma_l}
\end{matrix}\right).
\end{multline*}
The third determinant is equal to $0$ because its first row is equal to its 
second one multiplied by $\sqrt{\sigma_2|B_1|\over\sigma_1|B_2|}$. The second 
determinant can be written as
\begin{gather*}
\mathrm{det}\left(
\begin{matrix}
-{|B_1|\over\sigma_1}&-\sqrt{|B_1||B_2|\over\sigma_1\sigma_2}&\dots 
&-\sqrt{|B_1||B_l|\over\sigma_1\sigma_l}\\0& {1\over\sigma_2} & \dots &
0\\\vdots&\vdots&\ddots&\vdots\\-\sqrt{|B_l||B_1|\over\sigma_l\sigma_1}& 
-\sqrt{|B_l||B_2|\over\sigma_l\sigma_2}&\dots &{1-|B_l|\over\sigma_l}
\end{matrix}\right)={1\over\sigma_2}\left(M(1,3,4,\dots,l)-{1\over\sigma_1}M(3,4
,\dots,l)\right).
\end{gather*}
Finally using formula (\ref{minor}) for $k=l-1,l-2$ we obtain
\begin{multline*}
M(1,2,\dots,l)= 
{1\over\sigma_1}M(2,3,\dots,l)+{1\over\sigma_2}\left(M(1,3,4,\dots,l)-{
1\over\sigma_1}M(3,4,\dots,l)\right)=\\=
{1\over\sigma_1}{1-\suml_{j=\overline{2,l}} 
|B_j|\over\prod\limits_{j=\overline{2,l}} 
\sigma_j}+{1\over\sigma_2}\left({1-\suml_{j=\overline{1,l};j\not=2} 
|B_j|\over\prod\limits_{j=\overline{1,l};j\not=2} 
\sigma_j}-{1\over\sigma_1}{1-\suml_{j=\overline{3,l}} 
|B_j|\over\prod\limits_{j=\overline{3,l}} \sigma_j}\right)=\\=
{1-\suml_{j=\overline{2,l}} 
|B_j|-\suml_{j=\overline{1,l};j\not=2}|B_j|+\suml_{j=\overline{3,l}} |B_j|\over 
\prod\limits_{j=\overline{1,l}}\sigma_j }={1-\suml_{j=\overline{1,l}} |B_j|\over 
\prod\limits_{j=\overline{1,l}}\sigma_j }
\end{multline*}
and (\ref{minor}) is proved.

Now let us consider equation \eqref{mu_eq}. Multiplying it by 
$|B_{m+1}|\prod\limits_{j=1}^{m}(\sigma_j-\lambda)$ and then grouping the terms 
with the same exponents of $\lambda$  we obtain
\begin{gather}\label{mu_eq_new}
\suml_{k=0}^m \lambda^{m-k}A_k=0,
\end{gather}
where
\begin{gather}\label{mu_eq_new+}
A_k=(-1)^{m-k} \suml_{1\leq i_1<i_2<\dots<i_k\leq
m}\sigma_{i_1}\sigma_{i_2}\dots\sigma_{i_k}\left(|B_{m+1}|+|B_{i_1}|+|B_{i_2}
|+\dots+|B_{i_k}|
\right).
\end{gather}

Using  \eqref{minor}  and the equalities
\begin{gather*}
|B_{m+1}|+|B_{i_1}|+|B_{i_2}|+\dots+|B_{i_k}|=1-|B_{j_1}|-|B_{j_2}|-\dots 
-|B_{j_{m-k}}|,\quad
\sigma_{i_1}\sigma_{i_2}\dots\sigma_{i_k}={ \prod\limits_{j=1}^m \sigma_j\over 
\sigma_{j_1}\sigma_{i_2}\dots\sigma_{j_{m-k}}}
\end{gather*}
where 
$\{j_1,j_2,\dots,j_{m-k}\}=\{1,2,\dots,m\}\setminus\{i_1,i_2,\dots,i_{k}\}$ we 
can rewrite $A_k$ as
\begin{gather}\label{kt}
A_k=(-1)^m\left(\prod\limits_{j=1}^m \sigma_j\right)E_{m-k}\textcolor{black}{.}
\end{gather}

Finally we divide (\ref{mu_eq_new}) by $(-\lambda)^m \left(\prod\limits_{j=1}^m 
\sigma_j\right)$ and taking into account \eqref{MM}-\eqref{minor}, \eqref{kt} we 
get the equation \eqref{det}. Thus we have just proved that
\begin{gather}\label{lra}
\lambda\in 
\mathbb{R}\setminus\left(\{0\}\cup\cupl_{j=1}^m\{\sigma_j\}\right)\text{ is a 
root of \eqref{mu_eq}}\ \Longrightarrow\ \lambda\text{ is a root of 
\eqref{det}}.
\end{gather}
Since \eqref{mu_eq} has exactly $m$ roots $\mu_k\in 
\mathbb{R}\setminus\left(\{0\}\cup\cupl_{j=1}^m\{\sigma_j\}\right)$ and 
$\mu_k\not=\mu_l$ for $k\not=l$ then, obviously, \eqref{lra} implies 
\eqref{hmu}. Let us note that \eqref{lra} imply also  inequality \eqref{2po}. 

Lemma \ref{Nlm5} is proved which ends the proof of Theorem \ref{Nth}.
\end{proof}
\end{proof}

\subsection{\label{ss24}Asymptotics of the first $m$ eigenvalues of the 
Dirichlet Laplacian}

We get more complete information about the
behaviour of $\lambda_k^{D,\eps}$ , $k=1,\dots,m$.

\begin{theorem}\label{Dth}
For $k=1,\dots,m$ one has
\begin{gather}\label{Dthmain}
\liml_{\eps\to 0}\eps^{-2}{\lambda}_k^{D,\eps}=\sigma_{k}.
\end{gather}
\end{theorem}

\begin{remark}
We carry out the proof in the same way as in Theorem \ref{Nth}. But in the 
Dirichlet case the proof is simplified by the fact that the eigenfunctions 
corresponding to the first $m$ eigenvalues converge to $0$ in $B_{m+1}$ (i.e. 
$\mathbf{s}^k_{m+1}=0$, $k=1,\dots,m$). This observation leads to the functional 
$\mathcal{F}_0$ having more simple form comparing with the functional 
$\mathcal{F}$.
\end{remark}

\begin{proof}
Let $u_k\e$, $k\in\mathbb{N}$ be the eigenfunctions corresponding to 
$\lambda_k^{D,\eps}$ and satisfying the condition
\begin{gather}\label{Dort}
(u_k\e,u_l\e)_{L_2(Y\e)}=\delta_{kl},\\
u_k\e\text{ are real functions}.\notag
\end{gather}
For $\lambda_k^{D,\eps}$ the min-max principle looks as follows:
\begin{gather}\label{Dminimax1}
\forall k\in \mathbb{N}:\ \lambda_k^{D,\eps}=\minl_{u\in
\overset{\circ}{H}(u_1\e,\dots,u_{k-1}\e)}{\|\nabla
u\|^2_{L_2(Y\e)}\over \|u\|^2_{L_2(Y\e)}},
\end{gather}
where
\begin{gather}\label{Dhk}
\overset{\circ}{H}(u_1\e,\dots,u_{k-1}\e)=\left\{u\in H^1(Y\e):\
u|_{\partial Y}=0,\ (u,u_l\e)_{L_2(Y\e)},\  l=1,\dots,k-1\right\}.
\end{gather}
As in Theorem \ref{Nth} we conclude that there exist a subsequence (still 
denoted by $\eps$) and numbers $\mathbf{s}^k_{j}\in\mathbb{R}$, $k\in 
\mathbb{N}$, $j=1,\dots,m+1$ such that
\begin{gather}\label{Dbetadef}
\liml_{\eps\to 0}\langle u_{k}\e\rangle_{B_j}=\mathbf{s}^k_{j}.
\end{gather}
We denote
$\mathbf{s}^k=(\mathbf{s}^k_{1},\dots,\mathbf{s}^k_{m+1})\in\mathbb{R}^{m+1}$.
Below we will prove that $\mathbf{s}^k_{m+1}=0$ whenever $k\leq
m$.

\begin{lemma}\label{Dlm1}
One has for $k=1,\dots,m$:
\begin{gather}\label{Dlm1main}
\liml_{\eps\to 0}\lambda_k^{D,\eps}\leq C\eps^2.
\end{gather}
\end{lemma}

\begin{proof} As in Lemma \ref{Nlm1} we carry out the proof by induction.
For an arbitrary $s\in \mathbb{R}^{m+1}$ such that $s_{m+1}=0$ one
has:
\begin{gather*}
\lambda_1^{D,\eps}\leq {\|\nabla {U}\e(\cdot, s)\|^2_{L_2(Y\e)}\over  \| 
{U}\e(\cdot, s)\|^2_{L_2(Y\e)}}\leq C\eps^2
\end{gather*}
(recall that the function $U\e$ is defined by \eqref{Nvb}). Here we can use the 
min-max principle since $U\e(x,s)=0$ for $x\in\partial Y$ whenever $s_{m+1}=0$. 
Thus \eqref{Dlm1main} is valid for $k=1$.

Now, suppose that  \eqref{Dlm1main} is valid for $k\leq k'-1$ and
let us prove it for $k=k'$. One has the following Poincar\'{e} (for
$B_j$, $j=1,\dots,m$) and Friedrichs inequalities (for $B_{m+1}$):
\begin{gather}\label{DPo}j=1,\dots,m:\quad
\|u_k\e-\langle u_{k}\e\rangle_{B_j}\|^2_{L_2(B_j)}\leq  C\lambda_k^{D,\eps},\\
\label{DFr} \|u_k\e\|^2_{L_2(B_{m+1})}\leq  C\lambda_k^{D,\eps}.
\end{gather}
Using \eqref{DPo}, \eqref{DFr} and taking into account the validity of 
\eqref{Dlm1main} for $k\leq k'-1$ we get
\begin{gather}\label{Dconv1}
\forall l=1,\dots,k'-1:\begin{cases}u_{l}\e\underset{\eps\to 0}\to
\mathbf{s}^{l}_j\text{ strongly in }L_2(B_j),\ j=1,\dots,m,\\
u_{l}\e\underset{\eps\to 0}\to \mathbf{s}^{l}_{m+1}=0\text{
strongly in }L_2(B_{m+1}).
\end{cases}
\end{gather}

Let $\hat s\in\mathbb{R}^{m+1}\setminus\{0\}$ be an arbitrary
vector satisfying:
\begin{gather}\label{Dhats}
\hat{s}_{m+1}=0\text{ and }(\hat{s},\mathbf{s}^l)_{B}=0,\
l=1,\dots,k'-1.
\end{gather}
The choice of such a vector is always possible whenever $k'\leq m$. We denote
$$\hat u\e(x)=U\e(x,\hat s)-w\e(x,\hat s),\text{ where }
w\e(x,\hat s)=\suml_{l=1}^{k'-1}(U\e(\cdot,\hat
s),u_l\e(\cdot))_{L_2(Y\e)}u_l\e(x).
$$
Obviously $\hat u\e\in \overset{\circ}{H}(u_1\e,\dots,u_{k'-1}\e)$. In the same 
way as in Lemma \ref{Nlm1} (see \eqref{w1}-\eqref{w2}) we obtain
\begin{gather}\label{Dw1}
\liml_{\eps\to 0}\left(\|w\e(\cdot,\hat s)\|_{L_2(Y\e)}^2+ \eps^{-2} \|\nabla
w\e(\cdot,\hat s)\|_{L_2(Y\e)}^2\right)=0.
\end{gather}

Finally  using \eqref{Dminimax1} and taking into account
\eqref{Nvbest1}, \eqref{Nvbest3+}, \eqref{Dw1} we get
\begin{gather*}
\lambda_{k'}^{D,\eps}\leq {\|\nabla \hat u\e\|^2_{L_2(Y\e)}\over
\| \hat u\e\|^2_{L_2(Y\e)}}\sim {\|\nabla {U}\e(\cdot,\hat
s)\|^2_{L_2(Y\e)}\over  \|{U}\e(\cdot,\hat s)\|^2_{L_2(Y\e)}}\leq
C\eps^2
\end{gather*}
and \eqref{Dlm1main} is proved.

\end{proof}

\begin{lemma}\label{Dlm2}
One has for $k=1,\dots,m$:
\begin{gather}\label{Dlm2main}
\liml_{\eps\to 0}\eps^{-2}\lambda_k^{D,\eps}=F(\mathbf{s}^k).
\end{gather}
\end{lemma}

\begin{proof}
Using the Poincar\'{e} (for $B_j$, $j=1,\dots,m$) and Friedrichs (for
$B_{m+1}$) inequalities and taking into account Lemma \ref{Dlm1}
we conclude that
\begin{gather}\label{Dconv1}
\forall k=1,\dots,m:\ \begin{cases}u_{k}\e\underset{\eps\to 0}\to
\mathbf{s}^{k}_j\text{ strongly in }L_2(B_j),&  j=1,\dots,m\\
u_{k}\e\underset{\eps\to 0}\to \mathbf{s}^{k}_{m+1}=0\text{
strongly in }L_2(B_{m+1}).
\end{cases}
\end{gather}

As in Lemma \ref{Nlm2} we construct an approximation $\mathbf{u}_k^\eps$ for the
eigenfunction $u_k\e$ by the formula
\begin{gather*}
\mathbf{u}_k\e(x)=U\e(x,\mathbf{s}^k)-
\suml_{l=1}^{k-1}({U}\e(\cdot,\mathbf{s}^k),u\e_l(\cdot))_{L_2(Y\e)}u_l\e(x).
\end{gather*}
Since $\mathbf{s}^k_{m+1}=0$ then $\mathbf{u}_k\e(x)\in 
\mathrm{dom}(A^{D,\eps})$.
Repeating word-by-word the arguments of Lemma \ref{Nlm2}  we conclude that
\begin{gather*}
\eps^{-2}{\|\nabla \mathbf{u}_k\e\|^2_{L_2(Y\e)}}\sim
\eps^{-2}{\|\nabla U\e(\cdot,\mathbf{s}^k)\|^2_{L_2(Y\e)}}\sim
F(\mathbf{s}_k) \quad (\eps\to 0),
\end{gather*}
while $\delta_k\e:=u_k\e-\mathbf{u}_k\e$ brings vanishingly small contribution 
to $u_k\e$, namely
\begin{gather}\label{Dwprop}
\liml_{\eps\to 0}\eps^{-2}\|\nabla \delta_k\e\|^2_{L_2(Y\e)}=0.
\end{gather}
The lemma is proved.

\end{proof}

\begin{lemma}\label{Dlm3}

The vectors $\mathbf{s}^{k}$, $k=1,\dots,m$ satisfy the following conditions:
\begin{gather*}
F(\mathbf{s}^{k})\leq F(s)\text{ for any }s\in
\mathbf{H}_0({\mathbf{s}}^1,\dots,{\mathbf{s}}^{k-1}),
\end{gather*}
where
$$\mathbf{H}_0({\mathbf{s}}^1,\dots,{\mathbf{s}}^{k-1})=\{s\in\mathbb{R}^{m+1}:\ 
s_{m+1}=0,\ (s,s)_B=1,\ (s,{\mathbf{s}}^l)_B=0,\ l=\overline{1,k-1}\}.$$
\end{lemma}

\begin{proof}The lemma is proved similarly to Lemma \ref{Nlm3} taking into 
account that $\mathbf{s}_{m+1}^k=0$, $k=\overline{1,m}$.\end{proof}

We introduce the function $\mathcal{F}_0:\mathbb{R}^{m}\to\mathbb{R}$ by the 
formula
$$\mathcal{F}_0(q)=\suml_{j=1}^m \sigma_j
q_{j}^2.$$ We also introduce the vectors $\mathbf{q}^k\in
\mathbb{R}^{m}$ by the formula
$$\mathbf{q}_j^k=\mathbf{s}_j^k\sqrt{|B_j|},\ j=1,\dots,m.$$
Taking into account the equality $\mathbf{s}_{m+1}^k=0$
($k=1,\dots,m$) we can easily reformulate Lemmas
\ref{Dlm2}-\ref{Dlm3}.

\begin{corollary}\label{Dcor}
One has for $k=1,\dots,m$:
\begin{gather}\label{Dlmadd}
\liml_{\eps\to
0}\eps^{-2}\lambda_{k}^{D,\eps}=\mathcal{F}_0(\mathbf{q}^{k})\leq
\mathcal{F}_0(q)\text{ for any }q\in
\mathcal{H}_0({\mathbf{q}}^1,\dots,{\mathbf{q}}^{k-1}),
\end{gather}
where
$$\mathcal{H}_0({\mathbf{q}}^1,\dots,{\mathbf{q}}^{k-1})=\{q\in\mathbb{R}^{m}:\ 
(q,q)=1,\ (q,{\mathbf{q}}^l)=0,\ l=\overline{1,k-1}\}.$$
Here $(\cdot , \cdot)$ is the usual scalar product in $\mathbb{R}^{m}$.
\end{corollary}

Let us denote by  $E_\varkappa$ (here $\varkappa>0$ is a parameter) the 
$(m-1)$-dimensional ellipsoid
\begin{gather*}
E_\varkappa=\{q\in\mathbb{R}^{m}:\ 
\mathcal{F}_0(q)=\varkappa\}\textcolor{black}{.}
\end{gather*}
The numbers $\left\{ \sqrt{\varkappa\over{\sigma_j}}\right\}_{j=1}^m$ are the 
semi-axes of $E_\varkappa$.

We denote by $S$ the $(m-1)$-dimensional sphere
$$S=\{q\in\mathbb{R}^n:\ (q,q)=1\}\textcolor{black}{.}$$
It is clear that $E_\varkappa$ touches $S$ iff $\varkappa=\sigma_k$ for some 
$k$.

The ellipsoids $E_{\sigma_k}$, $k=1,\dots,m$ touch the sphere $S$ in two points 
$\tilde{\mathbf{q}}^k_{\pm}=\pm\underset{^{\overset{\qquad\quad\uparrow}{
\qquad\quad
k\text{-th place}}\qquad }}{(0,0,\dots,1,\dots,0)}$.

Using Corollary \ref{Dcor} via the same arguments as in the proof
of Theorem \ref{Nth} we obtain the following lemma.

\begin{lemma}\label{Dlm4}
One has for $k=1,\dots,m$:
\begin{gather*}
\mathbf{q}^k=\tilde{\mathbf{q}}^{k-1}_{-}\text{ or
}\mathbf{q}^k=-\tilde{\mathbf{q}}^{k-1}_{+}.
\end{gather*}
\end{lemma}

It follows from Corollary \ref{Dcor} and Lemma \ref{Dlm4} that
$\liml_{\eps\to 0}\eps^{-2}\lambda_{k}^{D,\eps}=\sigma_{k}$.
Theorem \ref{Dth} is proved.

\end{proof}

\subsection{\label{ss5}Asymptotics of the first eigenvalues of $\theta$-periodic 
Laplacian}

To complete the proof of Theorem \ref{th1} we also need the
information about the behaviour of the eigenvalues of the
operators $\mathcal{A}^{\theta,\eps}$. It turns out that their
behaviour is the same as the behaviour of the 
eigenvalues of either Neumann or Dirichlet Laplacians. Namely, the following 
theorem is valid.

\begin{theorem}\label{Tth}
Let
$$\theta_1=(1,1,\dots,1),\ \theta_2=-(1,1,\dots,1).$$

Then for $k=1,\dots,m$ one has
\begin{gather}\label{Tthmain1}
\theta=\theta_1:\ \liml_{\eps\to
0}\eps^{-2}{\lambda}_{k+1}^{\theta,\eps}=\mu_{k},\\\label{Tthmain2}
\theta=\theta_2:\ \liml_{\eps\to
0}\eps^{-2}{\lambda}_k^{\theta,\eps}=\sigma_{k}.
\end{gather}
\end{theorem}

\begin{remark}In fact it is possible to prove that \eqref{Tthmain2} is valid for 
an arbitrary $\theta\not=\theta_1$. However for the proof of Theorem \ref{th1} 
it is enough to prove \eqref{Tthmain2} only for $\theta=\theta_2$.
\end{remark}

\begin{proof}
The proof of \eqref{Tthmain1} is carried our word-by-word as the proof of 
Theorem \ref{Nth}. Indeed it is easy to see that when proving Theorem \ref{Nth} 
we have used only the following three facts that are specific for the Neumann 
boundary conditions:
\begin{itemize}
\item $\lambda_1^{N,\eps}=0$ and the corresponding eigenspace consists of 
constants,

\item There exists an orthonormal sequence of \textit{real} eigenfunctions of 
$A^{N,\eps}$,

\item For an arbitrary $s$ the function $U\e(\cdot,s)$ belongs to
$\mathrm{dom}(\mathcal{A}^{N,\eps})$.
\end{itemize}
However it is clear that all these properties are valid with $\theta_1$ instead 
of $N$.

The proof of \eqref{Tthmain2} is similarly to the proof of Theorem
\ref{Dth}. Indeed the proof of Theorem \ref{Dth} uses the
following three facts that are specific for the Dirichlet boundary
conditions:
\begin{itemize}

\item The Friedrichs inequality
\begin{gather}\label{Fried}
\|u\|^2_{L_2(B_{m+1})}\leq C_D\|\nabla u\|^2_{L_2(B_{m+1})},\ u\in 
\mathrm{dom}(\eta^{D,\eps})
\end{gather}
is valid. Here the constant $C_D>0$ is independent of $u$.

\item There exists an orthonormal sequence of \textit{real} eigenfunctions of 
$A^{N,\eps}$,

\item For an arbitrary $s\in \{s\in \mathbb{R}^{m+1}:\ s^{m+1}=0\}$ the function 
$U\e(\cdot,s)$ belongs to $\mathrm{dom}(\mathcal{A}^{D,\eps})$.

\end{itemize}
Inequality \eqref{Fried} with $\theta\not=\theta_1$ instead of $D$
was proved in \cite{Khrab3} for the case $m=1$. In the  case $m>1$
the proof is similar. Obviously, the remaining conditions are also
valid for $\theta_2$ instead of $D$.
\end{proof}

\subsection{\label{ss6} End of the proof of Theorem \ref{th1}}

It follows from
\eqref{AA} and \eqref{repres1} that
\begin{gather}\label{sp}
\sigma(\mathcal{A}\e)=\cupl_{k=1}^\infty [a_k^-(\eps),a_k^+(\eps)]
\end{gather}
where the compact intervals $[a_k^-(\eps),a_k^+(\eps)]$ are
defined by
\begin{gather}\label{interval}
[a_k^-(\eps),a_k^+(\eps)]= \cupl_{\theta\in \mathbb{T}^n}
\left\{\eps^{-2}\lambda_k^{\theta,\eps}\right\}.
\end{gather}

It follows from \eqref{enclosure} and \eqref{interval} that
\begin{gather}\label{double1}
\eps^{-2}\lambda_k^{N,\eps}\leq a_k^-(\eps)\leq
\eps^{-2}\lambda_k^{\theta_{1},\eps},\\\label{double2}
\eps^{-2}\lambda_k^{\theta_{2},\eps}\leq a_k^+(\eps)\leq
\eps^{-2}\lambda_k^{D,\eps}.
\end{gather}
Obviously if $k=1$ then the left and right-hand-sides of
\eqref{double1} are equal to zero. It follows from
\eqref{Nthmain}, \eqref{Tthmain1} that in the case $k=2,\dots,m+1$
they both converge to $\mu_{k-1}$ as $\eps\to 0$, while if $k>
m+1$ they converge to infinity in view of \eqref{lambdaN},
\eqref{lambdaT1}, \eqref{conv1}, \eqref{conv3}. Thus
\begin{gather}\label{a-}
a_1^-(\eps)=0,\quad \liml_{\eps\to 0}a_k^-(\eps)=\mu_{k-1}\text{
if }2\leq k\leq m+1,\quad \liml_{\eps\to
0}a_k^-(\eps)=\infty\text{ if }k>m+1.
\end{gather}
Similarly in view of \eqref{lambdaD}, \eqref{lambdaT2},
\eqref{conv2}, \eqref{conv3}, \eqref{Dthmain}, \eqref{Tthmain2}, \eqref{double2} 
one
has
\begin{gather}\label{a+}
\liml_{\eps\to 0}a_k^+(\eps)=\sigma_k\text{ if }1\leq k\leq
m,\quad  \liml_{\eps\to 0}a_k^+(\eps)=\infty\text{ if }k>m.
\end{gather}
Then \eqref{th1_f1} -\eqref{th1_f2} follow directly from
\eqref{sp}, \eqref{a-}, \eqref{a+}. Theorem \ref{th1} is proved.

\section*{Acknowledgements} The author is grateful to Prof. E. Khruslov
for the helpful discussions. This work is supported by the German
Research Foundation (DFG) through Research Training Group 1294
"Analysis, Simulation and Design of Nanotechnological Processes".


\bibliographystyle{elsarticle-num}

\end{document}